\documentclass[a4paper,10pt, oneside, reqno]{amsart}
\pdfoutput=1

\usepackage[activate={true,nocompatibility},final,tracking=false,kerning=true,spacing=true,factor=1100,stretch=0,shrink=0]{microtype}

\usepackage[T1]{fontenc} 
\usepackage[utf8]{inputenc}
\usepackage[british]{babel}
\usepackage{mathrsfs}
\usepackage{amsfonts}
\usepackage{amsmath} 
\usepackage{amsthm}  
\usepackage{amssymb}
\usepackage{mathtools}
\mathtoolsset{showonlyrefs}
\usepackage{graphicx}
\usepackage{xcolor}
\usepackage{url}
\usepackage[noadjust]{cite}
\usepackage{xifthen}
\usepackage{stmaryrd}
\usepackage{centernot}
\usepackage{indentfirst}
\usepackage[hang,flushmargin]{footmisc}
\usepackage{geometry}
\usepackage{hyperref}
\usepackage{tikz}
\usetikzlibrary{matrix,arrows,decorations.pathmorphing}

\usepackage{enumitem}
\setlist[1]{labelindent=\parindent}
\setlist[enumerate,1]{label = \arabic*.,ref   = \arabic*,wide, labelwidth=!, labelindent=0pt}
\setlist[enumerate,2]{label = \alph*),ref   = \theenumi\alph*)}
\setlist[enumerate,3]{label = \roman*),ref   = \theenumii.\roman*}

\newcommand{\from}{\colon}
\newcommand{\into}{\hookrightarrow}
\newcommand{\implica}{\rightarrow}
\newcommand{\coimplica}{\leftrightarrow}
\renewcommand{\phi}{\varphi}
\renewcommand{\epsilon}{\varepsilon}
\renewcommand{\models}{\vDash}
\newcommand{\proves}{\vdash}
\newcommand{\pfin}{{\mathscr P}_{<\omega}}
\newcommand{\restr}{\upharpoonright}
\newcommand{\monster}{\mathfrak U}
\newcommand{\smallsubset}{\mathrel{\subset^+}}

\newcommand{\smallprec}{\mathrel{\prec^+}}

\newcommand{\invtypes}{S^{\mathrm{inv}}}
\newcommand{\invext}{\mid}
\newcommand{\bla}[4]{{#1}_{#2}#3\ldots#3{#1}_{#4}}
\newcommand{\wort}{\mathrel{\perp^{\!\!\mathrm{w}}}}
\newcommand{\nwort}{\mathrel{{\centernot\perp}^{\!\!\mathrm{w}}}}
\newcommand{\invtilde}{\operatorname{\widetilde{Inv}}(\monster)}
\newcommand{\invtildeof}[1]{\operatorname{\widetilde{Inv}}({#1})}

\newcommand{\invtildestarof}[2]{\operatorname{\widetilde{Inv}}_{#1}({#2})}
\newcommand{\invbar}{\operatorname{\overline{Inv}}(\monster)}

\newcommand{\domeq}{\mathrel{\sim_\mathrm{D}}}
\newcommand{\doms}{\mathrel{\ge_\mathrm{D}}}
\newcommand{\ndoms}{\mathrel{{\centernot\ge}_{\mathrm{D}}}}

\newcommand{\inverse}{^{-1}}
\newcommand{\allora}{\Rightarrow}

\newcommand{\eq}{\mathrm{eq}}

\newcommand{\RV}{\mathrm{RV}}
\newcommand{\K}{\mathrm{K}}
\newcommand{\key}{\mathrm{k}}
\newcommand{\Gam}{\mathrm{\Gamma}}
\newcommand{\RVses}{\mathcal{RV}}
\newcommand{\primo}{\mathfrak{p}}
\newcommand{\constants}{C}
\newcommand{\icsg}{\operatorname{CS}^\mathrm{inv}}

\DeclareMathOperator{\Th}{Th}
\DeclareMathOperator{\tp}{tp}
\DeclareMathOperator{\dcl}{dcl}
\DeclareMathOperator{\acl}{acl}
\DeclareMathOperator{\aut}{Aut}

\DeclareMathOperator{\trdeg}{trdeg}
\DeclareMathOperator{\rv}{rv}
\DeclareMathOperator{\Kof}{K}
\DeclareMathOperator{\keyof}{k}
\DeclareMathOperator{\RVof}{RV}
\DeclareMathOperator{\Gamof}{\mathrm{\Gamma}}
\DeclareMathOperator{\RVsesof}{\mathcal{RV}}

\DeclareMathOperator{\sort}{Q}
\renewcommand{\div}{\operatorname{div}}
\DeclarePairedDelimiter{\set}{\{}{\}}
\DeclarePairedDelimiter{\abs}{\lvert}{\rvert}
\DeclarePairedDelimiter{\class}{\llbracket}{\rrbracket}

\DeclarePairedDelimiter{\Span}{\langle}{\rangle}

\theoremstyle{definition}
\newtheorem{defin}{Definition}[section]

\newtheorem{thm}[defin]{Theorem}
\newtheorem{pr}[defin]{Proposition}
\newtheorem{co}[defin]{Corollary}
\newtheorem{lemma}[defin]{Lemma}
\newtheorem{notation}[defin]{Notation}
\newtheorem{eg}[defin]{Example}
\newtheorem{rem}[defin]{Remark}
\newtheorem{fact}[defin]{Fact}
\newtheorem{ass}[defin]{Assumption}

\newtheorem{question}[defin]{Question}
\newtheorem{prob}[defin]{Problem}

\newtheorem*{claim}{Claim}
\theoremstyle{plain}
\newtheorem{alphthm}{Theorem}

\let\oldqed\qedsymbol
\newcommand{\qedclaim}{\mbox{$\underset{\textsc{claim}}{\square}$}}
\makeatletter
\newenvironment{claimproof}[1][\it Proof of Claim]{
\let\qedsymbol\qedclaim
  \par
  \pushQED{\qed}%
  \normalfont \topsep6\p@\@plus6\p@\relax
  \trivlist
\item[\hskip\labelsep
  \upshape
  #1\@addpunct{.}]\ignorespaces
}{%
  \popQED\endtrivlist\@endpefalse
}
\makeatother
\let\qedsymbol\oldqed

\author{Martin Hils}
\address{Institut f\"ur Mathematische Logik und Grundlagenforschung, Universit\"at M\"unster, Einsteinstr. 62, D-48149 M\" unster, Germany}
\email{hils@uni-muenster.de}
\author{Rosario Mennuni}
\address{Universit\`a di Pisa, Dipartimento di Matematica, Largo Bruno Pontecorvo 5, 56127, Pisa, Italy}
\email{R.Mennuni@posteo.net}
\thanks{Both authors were supported by the Deutsche Forschungsgemeinschaft (DFG)  via HI 2004/1-1 (part of the ANR-DFG project GeoMod) and under Germany’s Excellence Strategy EXC 2044-390685587, Mathematics Münster: Dynamics-Geometry-Structure. RM was supported by the Italian research project  PRIN 2017: ``Mathematical Logic: models, sets, computability'' Prot.~2017NWTM8RPRIN and is a member of the INdAM research group GNSAGA}
\title{The domination monoid in henselian valued fields}
\makeatletter
\@namedef{subjclassname@2020}{\textup{2020} MSC} 
\makeatother

\begin{document}
\begin{abstract}
We study the domination monoid in various classes of structures arising from henselian valuations, including  $\RVses$-expansions of henselian valued fields of equicharacteristic $0$ (and, more generally, of benign valued fields), $\primo$-adically closed fields, monotone D-henselian differential valued fields with many constants, regular ordered abelian groups, and pure short exact sequences of abelian structures. We obtain Ax--Kochen--Ershov type reductions to suitable fully embedded families of sorts in quite general settings, and full computations in concrete ones. 
\end{abstract}

\maketitle
In their seminal work~\cite{hhm} on stable domination, Haskell, Hrushovski and Macpherson introduced the \emph{domination monoid} $\invtilde$, and showed that in algebraically closed valued fields it decomposes as $\invtildeof{\keyof(\monster)}\times \invtildeof{\Gamof(\monster)}$, where $\key$ denotes the residue field and $\Gam$ the value group.\footnote{Strictly speaking, \cite{hhm} works with $\invbar$, which is in general different, but coincides with $\invtilde$ in their setting. See~\cite[Remark~2.1.14 and Theorem~5.2.22]{mennuni_thesis}.} A similar result was proven in~\cite{ehm,dominomin} in the case of real closed fields with a convex valuation. This paper revolves around understanding  $\invtilde$ in more general classes of valued fields, and expansions thereof. A special case of our results is the following.
\begin{alphthm}[{Corollary~\ref{co:benign}}]\label{athm:bendec}
Let $T$ be the theory of a henselian valued field of equicharacteristic $0$, or algebraically maximal Kaplansky, possibly enriched on $\key$ and $\Gamma$.  If all  $\key^\times/(\key^\times)^n$ are finite, then $\invtilde\cong\invtildeof{\keyof(\monster)}\times \invtildeof{\Gamof(\monster)}$.
\end{alphthm}
More generally, we obtain a two-step reduction, first to leading term structures, and then, using technology on pure short exact sequences recently developed in~\cite{acgz20}, to $\key$ and $\Gamma$, albeit in a form which, in general, is (necessarily) slightly more involved. We also compute $\invtildeof{\Gamof(\monster)}$ when the  theory of $\Gamma$ has an archimedean model, and prove several accessory statements. 

Before stating our results in more detail, let us give an informal account of the context (see Section~\ref{sec:prelim} for the precise definitions).
The starting point is the space $\invtypes(\monster)$ of \emph{invariant types} over a monster model $\monster$: those which are invariant over a small subset. It is a dense subspace of $S(\monster)$, whose points may be canonically extended to larger parameter sets. Such extensions allow to define the \emph{tensor product}, or \emph{Morley product}, obtaining a semigroup $(\invtypes(\monster), \otimes)$, in fact a monoid. The space $\invtypes(\monster)$ also comes with a preorder $\doms$, called \emph{domination}: roughly, $p\doms q$ means that $q$ is recoverable from $p$ plus a small amount of information. The quotient by the induced equivalence relation, \emph{domination-equivalence} $\domeq$, is then  a poset, denoted by $(\invtilde, \doms)$. If $\otimes$ respects $\doms$, i.e.~if $(\invtypes(\monster), \otimes, \doms)$ is a preordered semigroup, then $\domeq$ is a congruence with respect to $\otimes$ and we say that \emph{the domination monoid is well-defined}, and equip $(\invtilde, \doms)$ with the operation induced by $\otimes$. Compatibility of $\otimes$ and $\doms$ in a given theory can be shown by using certain sufficient criteria, isolated in~\cite{invbartheory} and applied e.g.~in~\cite{wbdmt}, or by finding a nice system of representatives for $\domeq$-classes (cf.~Proposition~\ref{pr:comptransfer}). Nevertheless, in general, $\otimes$ may fail to respect $\doms$~\cite{invbartheory}. Hence, when dealing with $\invtilde$ in a given structure, one needs to understand whether it is well-defined as a monoid; and, when dealing with it in the abstract, the monoid structure cannot be taken for granted.

Recall that to a valued field $\K$ are associated certain abelian groups augmented by an absorbing element, fitting in a short exact sequence $1\to (\key, \times)\to (\K,\times)/(1+\mathfrak m)\to \Gam\cup\set\infty\to  0$, denoted by $\RVses$.
This sequence is interpretable in $\K$, and this interpretation endows it with extra structure. The amount of induced structure clearly depends on whether $\K$ has extra structure itself, but at a bare minimum $\key$ will carry the language of fields and $\Gam$ that of ordered abelian groups. 
By~\cite{Basarab_1991} (see also~\cite{Kuh}, or~\cite{flenner,hhqe} for a more modern treatment), henselian valued fields of residue characteristic $0$ eliminate quantifiers relatively to $\RVses$, and the latter is fully embedded with the structure described above. This holds resplendently, in the sense that it is still true after arbitrary expansions of $\RVses$. The same holds in the algebraically maximal Kaplansky case, by~\cite{Kuh} (see also~\cite{hhqe}).\footnote{Note that these quantifier elimination results are already implicitly contained in~\cite{Del-These}.} These are known after~\cite{touchardburden} as classes of  \emph{benign} valued fields and, in several contexts, they turn out to be particularly amenable to model-theoretic investigation. One of our main results says the context of domination is no exception.
\begin{alphthm}[{Theorem~\ref{thm:rvred}}]\label{athm:rv}
In every $\RVses$-expansion of a benign theory of valued fields there is an isomorphism of posets $\invtilde\cong\invtildeof{\RVsesof(\monster)}$.
  If $\otimes$ respects $\doms$ in $\RVsesof(\monster)$, then $\otimes$ respects $\doms$ in $\monster$, and the above is an isomorphism of monoids.
\end{alphthm}
Having reduced $\invtilde$ to the short exact sequence $\RVses$, the next step is to reduce it to its kernel $\key$ and quotient $\Gam$. If we add an angular component map, the sequence $\RVses$ splits and we obtain a product decomposition as in Theorem~\ref{athm:bendec} (Remark~\ref{rem:denefpas}). Without an angular component, a product decomposition is not always possible; yet, $\key$ and $\Gamma$ still exert a tight control on $\RVses$.  This behaviour is not peculiar of $\RVses$: it holds in short exact sequences of abelian structures, provided they satisfy a purity assumption, 
using the relative quantifier elimination from~\cite{acgz20}. For reasons to be clarified later (Remark~\ref{rem:wesast}), here it is natural to look at types in infinitely many variables, say $\kappa$, hence at the corresponding analogue $\invtildestarof\kappa\monster$ of $\invtilde$.
\begin{alphthm}[Corollary~\ref{co:decLtps}]\label{athm:ses}
  Let $\monster$ be a pure short exact sequence $0\to \mathcal{A}\to \mathcal{B}\to \mathcal{C}\to 0$ of $L$-abelian structures, where  $\mathcal{A}$ and $\mathcal{C}$ may carry extra structure. Let $\kappa\ge \abs L$ be a small cardinal. There is an expansion $\mathcal{A}_\mathcal F$ of $\mathcal{A}$ by imaginary sorts yielding an isomorphism of posets
    $\invtildestarof{\kappa}\monster\cong \invtildestarof{\kappa}{\mathcal{A}_\mathcal F(\monster)}\times\invtildestarof{\kappa}{\operatorname{\mathcal C}(\monster)}$.
  If $\otimes$ respects $\doms$ in both $\mathcal{A}_\mathcal F(\monster)$ and $\operatorname{\mathcal C}(\monster)$, then $\otimes$ respects $\doms$ in $\monster$, and the above is an isomorphism of monoids.
\end{alphthm}
In algebraically or real closed valued fields, the isomorphism $\invtilde\cong\invtildeof{\keyof(\monster)}\times \invtildeof{\Gamof(\monster)}$ is complemented by a computation of the factors, carried out in~\cite{hhm,dominomin}. In particular, if $\Gamof(\monster)$ is divisible, then $\invtildeof{\Gamof(\monster)}$ is isomorphic to the upper semilattice of finite sets of invariant convex subgroups of $\Gamof(\monster)$ (in the sense of Definition~\ref{defin:invconvsbgrs}). A further contribution of this work is the computation of $\invtilde$ in the next simplest class of theories of ordered abelian groups: those with an archimedean model, known as \emph{regular}.  Denote by $\icsg(\monster)$ the set of invariant convex subgroups of $\monster$,  by $\mathscr P_{\le\kappa}(\icsg(\monster))$ the upper semilattice of its subsets of size at most $\kappa$, and by $\hat \kappa$  the ordered monoid of cardinals smaller or equal than $\kappa$ with cardinal sum.

\begin{alphthm}[{Corollary~\ref{co:isro}}]\label{athm:oagsdom}
Let $T$ be the theory of a regular ordered abelian group, $\kappa$ a small infinite cardinal, and  $\mathbb P_T$ the set of primes $\primo$ such that $\monster/\primo \monster$ is infinite. Then  $\invtildestarof\kappa{\monster^\eq}$ is well-defined, and
    $\invtildestarof\kappa{\monster^\eq}\cong \mathscr P_{\le\kappa}(\icsg(\monster))\times \prod_{\mathbb P_T} \hat \kappa$.
\end{alphthm}
Theorem~\ref{athm:oagsdom} applies to Presburger Arithmetic, the theory of $(\mathbb Z,+,<)$. Pairing this 
with a suitable generalisation of Theorem~\ref{athm:rv}, we obtain the following.
\begin{alphthm}[{Corollary~\ref{co:qp}}]\label{athm:qp}
  In the theory $\Th(\mathbb Q_\primo)$ of $\primo$-adically closed fields, $\otimes$ respects $\doms$, and $\invtilde\cong\mathscr P_{<\omega}(\icsg(\Gamof(\monster)))$.
  
\end{alphthm}
 A similar statement (Corollary~\ref{co:wittv}) holds for Witt vectors over $\mathbb F_\primo^\mathrm{alg}$.
 Finally, we move to monotone D-henselian differential valued fields with many constants. While Theorem~\ref{athm:rv} does not generalise to this context (Remark~\ref{rem:vdfnofin}),  its analogue for $\invtildestarof\kappa\monster$ does (Theorem~\ref{thm:diff}). We fully compute $\invtildestarof\kappa\monster$ in the model companion $\mathsf{VDF}_\mathcal{EC}$. Similar results hold for $\sigma$-henselian valued difference fields (Remark~\ref{rem:vdifff}).
\begin{alphthm}[{Theorem~\ref{thm:vdfec}}]\label{athm:vdf}
  In  $\mathsf{VDF}_\mathcal{EC}$, for every small infinite cardinal $\kappa$, the monoid $\invtildestarof\kappa\monster$ is well-defined, and we have isomorphisms
  \[
    \invtildestarof\kappa\monster\cong \invtildestarof\kappa{\keyof(\monster)}\times \invtildestarof\kappa{\Gamof(\monster)}\cong \prod_{\delta(\monster)}^{\le\kappa} \hat \kappa\times \mathscr P_{\le\kappa}(\icsg(\Gamof(\monster)))
  \]
  where  $\delta(\monster)$ is a certain cardinal, and $\prod_{\delta(\monster)}^{\le\kappa} \hat \kappa$ denotes the submonoid of $\prod_{\delta(\monster)} \hat \kappa$ consisting of $\delta(\monster)$-sequences with support of size at most $\kappa$.
\end{alphthm}

The paper is structured as follows. In the first two sections we recall some preliminary notions and facts, and deal with some easy observations about orthogonality of invariant types. In Section~\ref{sec:ro} we prove Theorem~\ref{athm:oagsdom}, while in Section~\ref{sec:ses} we study expanded pure short exact sequences of abelian structures, proving Theorem~\ref{athm:ses}. The results from these two sections are then combined in Section~\ref{sec:owfmig} to deal with the case of ordered abelian groups with finitely many definable convex subgroups. In Section~\ref{sec:bvf} we prove Theorem~\ref{athm:rv}, and illustrate how it may be combined with Theorem~\ref{athm:ses} to obtain statements such as Theorem~\ref{athm:bendec}.  Section~\ref{sec:mixedchar} deals with finitely ramified mixed characteristic henselian valued fields and includes a proof of Theorem~\ref{athm:qp}, and  Section~\ref{sec:diff} deals with the differential case, proving Theorem~\ref{athm:vdf}. 

\section*{Acknowledgements}
RM thanks E.~Hrushovski for the useful discussions around orthogonality of invariant types. We thank A.~Gehret for Remark~\ref{rem:gehretqe}, and the anonymous referee for providing extensive and thorough feedback that helped improve our paper.

\section{Preliminaries}\label{sec:prelim}
\subsection{Notation and conventions}
We adopt the conventions and notations of~\cite[Subsection~1.1]{dominomin} (e.g. we usually (and tacitly) fix  a \emph{monster model} $\monster$, and \emph{definable} means $\monster$-definable), with the following additions and differences.
The set of prime natural numbers is denoted by $\mathbb P$.  Sorts are denoted by upright letters, as in $\mathrm{A}, \K, \key, \Gamma$,  families of sorts by calligraphic letters such as $\mathcal C$, and $S_{{\mathcal C}^{<\omega}}(A)$ stands for the disjoint union of all spaces of types in finitely many variables, each with sort in $\mathcal{C}$. Terms may contain parameters, as in $t(x,d)$; we write $t(x)$ if they do not.

\subsection{Domination}
We assume familiarity with invariant types, and recall some basic definitions and facts about domination. See~\cite[Subsection~1.2]{dominomin}, \cite[Subsection~2.1.2]{mennuni_thesis}, and \cite{invbartheory} for a more thorough treatment.

If $p(x), q(y)\in S(\monster)$, let $S_{pq}(A)$ be the set of    types over $A$ in variables $xy$ extending
    $(p(x)\restr A)\cup (q(y)\restr A)$.   We say that $p(x)\in S(\monster)$ \emph{dominates} $q(y)\in S(\monster)$, and write $p\doms q$, iff there are a small $A\smallsubset \monster$ and 
    $r\in S_{pq}(A)$ such that $p(x)\cup r(x,y)\proves q(y)$. 
 We say that $p,q\in S(\monster)$ are \emph{domination-equivalent}, and write $p\domeq q$, iff $p\doms q$ and $q\doms p$. We denote the
    domination-equivalence class of $p$ by $\class p$.
 The \emph{domination poset} $\invtilde$ is the quotient of $\invtypes(\monster)$ by $\domeq$, equipped with the partial order induced by $\doms$, denoted by the same symbol.  In other words, domination is the semi-isolation counterpart to $\mathrm{F}^\mathrm{s}_{\kappa(\monster)}$-isolation in the sense of~\cite[Chapter IV]{classificationtheoryshelah}; the two notions are distinct, see~\cite[Example~3.3]{wbdmt}.

 We will be mostly concerned with domination on $\invtypes(\monster)$. When describing a witness to $p\doms q$, we  write e.g.~``let $r$ contain $\phi(x,y)$'' with the meaning ``let $r\in S_{pq}(A)$ contain $\phi(x,y)$, for an $A$ such that $p,q\in \invtypes(\monster, A)$''.
By \cite[Lemma~1.14]{invbartheory}, if $p_0, p_1\in\invtypes(\monster)$  and $p_0\doms p_1$, then $p_0\otimes q\doms p_1\otimes q$.
  We say that $\otimes$ \emph{respects} $\doms$ iff $q_0\doms q_1$ implies $p\otimes q_0\doms p\otimes q_1$. If this is the case, the \emph{domination monoid} is the expansion of $\invtilde$ by the operation induced by $\otimes$, also denoted by $\otimes$. If we say \emph{$\invtilde$ is well-defined} (as a partially ordered monoid) we mean ``$\otimes$ respects $\doms$''.  As $\invtilde$ is always well-defined as a poset, this should cause no confusion.

 Adding imaginary sorts to $\monster$ may result in an enlargement of $\invtilde$~\cite[Corollary~3.8]{invbartheory}. Yet, if $T$ eliminates imaginaries, even just geometrically, then the natural embedding $\invtilde\into \invtildeof{\monster^\eq}$ is easily seen to be an isomorphism.
By~\cite[Proposition~1.23]{invbartheory}, domination witnessed by algebraicity  is  compatible with $\otimes$:  if $p,q_0,q_1\in\invtypes(\monster)$ and, for $i<2$, there are realisations $a_i\models q_i$ such that $a_1\in \acl(\monster a_0)$, then for all invariant $p$ we have $p\otimes q_0\doms p\otimes q_1$. In particular, if $T$ has geometric elimination of imaginaries, then $\invtildeof{\monster^\eq}$ is well-defined if and only if $\invtilde$ is.

 Frequently, we will equip a family of sorts, say $\mathcal{A}=\set{\mathrm{A}_s\mid s\in S}$, with the traces of some $\emptyset$-definable relations, and consider it as a standalone structure. We call $\mathcal A$  \emph{fully embedded} iff, for each $\bla s0,n\in S$, every subset of $(\bla {\mathrm{A}}{s_0}\times{s_n})(\monster)$ is definable in $\monster$ if and only if it is definable in $\operatorname{\mathcal{A}}(\monster)$.  When talking of a fully embedded $\mathcal A$ in the abstract, as below, we assume a structure on $\mathcal A$ to be fixed.

\begin{fact}[{\cite[Proposition~2.3.31]{mennuni_thesis}}]\label{fact:fullembemb}
Let $\mathcal{A}$ be a fully embedded family of sorts, and let $\iota: S_{\mathcal A^{<\omega}}(\operatorname{\mathcal{A}}(\monster))\to S(\monster)$ send a type of $\operatorname{\mathcal{A}}(\monster)$ to the unique type of $\monster$ it entails.
The type $p$ is invariant if and only if $\iota(p)$ is. The map $\iota\restr \invtypes(\operatorname{\mathcal{A}}(\monster))$ is an injective $\otimes$-homomorphism  inducing an embedding of posets $\tilde\iota\from\invtildeof{\operatorname{\mathcal A}(\monster)}\into \invtilde$ which,  if $\otimes$ respects $\doms$ in $\monster$ (hence also in $\operatorname{\mathcal A}(\monster)$),  is also an embedding of monoids.
\end{fact}

\begin{rem}\label{Rem:invimply}With the notation and assumptions from Fact~\ref{fact:fullembemb}, if $p$ is an invariant  $\operatorname{\mathcal{A}}(\monster)$-type, $\monster_1\succ \monster$, and $\operatorname{\mathcal{A}}(\monster)\subseteq B\subseteq \operatorname{\mathcal{A}}(\monster_1)$, then $(p\invext B)\vdash (\iota p\invext \monster B)$.
\end{rem}

\begin{proof}Suppose $\phi(x,w,t)\in L(\emptyset)$, $d\in \monster$,  $e\in B$, and $\iota p(x)\invext B\proves \phi(x,d,e)$. Since $x,t$ are $\mathcal A$-variables, and $d\in \monster$, full embeddedness yields an $L_\mathcal A(\operatorname{\mathcal A}(\monster))$-formula $\psi(x,t)$ equivalent to $\phi(x,d,t)$. So $\psi(x,e)\in p\invext B$ and we are done.
\end{proof}

\begin{pr}\label{pr:comptransfer}
Assume for all $p\in \invtypes(\monster)$ there is a tuple  $\tau^p$ of definable functions with codomains in a fully embedded $\mathcal A$ such that $p\domeq \tau^p_*p$ and  $p\otimes q\domeq \tau^p_*p\otimes \tau^q_*q$. If $\otimes$ respects $\doms$ in $\operatorname{\mathcal A}(\monster)$, then $\otimes$ respects $\doms$ in $\monster$. 
\end{pr}
\begin{proof}
We need to show that if $q_0\doms q_1$ then $p\otimes q_0\doms p\otimes q_1$.  By assumption,  $p\otimes q_0\domeq \tau^p_*p\otimes \tau^{q_0}_*q_0$ and $\tau^p_*p\otimes \tau^{q_1}_*q_1\domeq p\otimes q_1$. Since  $\otimes$ respects $\doms$ in $\operatorname{\mathcal A}(\monster)$, we obtain $\tau^p_*p\otimes \tau^{q_0}_*q_0\doms\tau^p_*p\otimes \tau^{q_1}_*q_1$, and we are done.
\end{proof}
Note that a map $\tau$ as above induces an inverse of $\tilde \iota$.

\subsection{A word on \textasteriskcentered-types}\label{subsec:startypes}
 We will deal with types in a small infinite number of variables, also known in the literature as \textasteriskcentered-types. We define $\invtildestarof\kappa\monster$ as the quotient of $S_{<\kappa^+}(\monster)$ by $\domeq$. Note that, by padding with realised coordinates and permuting variables, every $\domeq$-class has a representative with variables indexed by $\kappa$. We leave to the reader easy tasks such as defining the $\alpha$-th power $p^{(\alpha)}$, for $\alpha$ an ordinal, or such as convincing themself that basic statements such as Fact~\ref{fact:fullembemb} generalise.

Nevertheless, it is not clear if well-definedness of $\invtilde$ implies well-definedness of $\invtildestarof\kappa\monster$ (the converse is easy): for instance, at least a priori, one could have a situation where the finitary $\invtilde$ is well-defined, but there are $1$-type $q_0$ and a $\kappa$-type $q_1$ such that $q_0\doms q_1$ but, for some $p$, we have $p\otimes q_0\ndoms p\otimes q_1$. In the rest of the paper we will say e.g.~``$\otimes$ respects $\doms$'' with the understanding that, whenever \textasteriskcentered-types are involved, this is to be read as ``$\otimes$ respects $\doms$ on \textasteriskcentered-types''.
 \begin{question}\label{q:star}
If $\otimes$ respects $\doms$ on finitary types, does $\otimes$ 
respect $\doms$ on \textasteriskcentered-types?
\end{question}

\section{Orthogonality}
\begin{defin} We say that $p,q\in S(A)$ are \emph{weakly orthogonal}, and write $p\wort q$, iff $p(x)\cup q(y)$ implies a complete $xy$-type over $A$.
 We say that $p,q\in \invtypes(\monster)$ are \emph{orthogonal}, and write $p\perp q$, iff $(p\invext B)\wort (q\invext B)$ for every $B\supseteq \monster$.
 Two definable sets $\phi, \psi$ are \emph{orthogonal} iff for every $n,m\in \omega$, every $p\in S_{\phi^n}(\monster)$ and $q\in S_{\psi^m}(\monster)$, we have $p\wort q$.
 Two families of sorts $\mathcal{A}$, $\mathcal{C}$  are \emph{orthogonal} iff every cartesian product of sorts in $\mathcal{A}$ is orthogonal to every cartesian product of sorts in $\mathcal{C}$.
\end{defin}
It is easily seen that if $p,q\in \invtypes(\monster, M)$ are weakly orthogonal and $\monster_1\succ\monster$ is $\abs  M^+$-saturated and $\abs  M^+$-strongly homogeneous, then $(p\invext \monster_1)\wort(q\invext \monster_1)$.  This can fail for arbitrary $B\supseteq \monster$, i.e.~weak orthogonality is indeed weaker than orthogonality. While this is folklore (the second author thanks E.~Hrushovski for pointing this out), we could not find any example in print, so we  record one.
 \begin{eg}
There is a theory with invariant $p,q$ such that  $p\wort q$ but $p\centernot\perp q$.
 \end{eg}
\begin{proof}
Let $L$ be a two-sorted language with sorts $\mathrm{P},\mathrm{O}$ (points, orders) and a relation symbol $x<_ty$ of arity $\mathrm{P}^2\times \mathrm{O}$. The class $K$ of  finite $L$-structures where, for every $d\in \mathrm{O}$, the relation $x<_dy$ is a linear order, is a (strong) amalgamation class. Let $T$ be the theory of the  Fra\"iss\'e limit of $K$. Fix a small  $M\models T$, and let $p,q$ be the $1$-types of sort $\mathrm{P}$ defined as
$p(x)=\set{m<_dx<_de\mid d\in \operatorname{O}(\monster), m\in M, e\in \operatorname{P}(\monster), e>M }$ and $q(y)\coloneqq\set{e<_dy\mid d\in \operatorname{O}(\monster), e\in \operatorname{P}(\monster)}$.
By quantifier elimination $p$, $q$ are complete, $p$ is $M$-invariant, and  $q$ is $\emptyset$-definable, hence $\emptyset$-invariant.

Since $M$ is small, for every $d\in \operatorname{O}(\monster)$ it is $<_d$-bounded, hence   $p\wort q$. Let $b$ be a point of sort $\mathrm{O}$ such that $M$ is $\le_b$-cofinal in $\monster$, and set $B\coloneqq \monster b$. Then $(q(y)\invext B)\proves y\ge_b \operatorname{P}(\monster)$ and $(p(x)\invext B)\proves x\ge_b \operatorname{P}(\monster)$, and both $x<_by$ and $y<_bx$  are consistent with $(p(x)\invext B)\cup (q(y)\invext B)$, which is therefore not complete.
 \end{proof}
\begin{rem}\label{rem:wort}
If $p\in S(A)$ is such that $p\wort p$, then $p$ is realised in $\dcl(A)$. 
 If $p,q\in \invtypes(\monster)$ and $p\wort q$, then $p(x)\otimes q(y)=q(y)\otimes p(x)$: they both coincide with (the unique completion of) $p(x)\cup q(y)$.
 Two definable sets $\phi, \psi$ are orthogonal if and only if every definable subset of $\phi^m(x)\land \psi^n(y)$ can be defined by a finite disjunction of formulas of the form $\theta(x)\land \eta(y)$.
If two $M$-definable sets are orthogonal, then the definition of orthogonality still holds after replacing $\monster$ with $M$.
Adding imaginaries preserves orthogonality, in the following sense.
Let $\mathcal{A}$ be a family of sorts, and let $\tilde{\mathcal{A}}$ be a larger family, consisting of $\mathcal{A}$ together with imaginary sorts obtained as definable quotients of products of elements of $\mathcal{A}$. Let $\tilde{\mathcal{C}}$ be obtained similarly from another family of sorts $\mathcal{C}$. If $\mathcal{A}$ and $\mathcal{C}$ are orthogonal, then so are $\tilde{\mathcal{A}}$ and $\tilde{\mathcal{C}}$.
\end{rem}
By~\cite[Proposition~3.13]{invbartheory}, if 
$p_0\doms p_1$ and $p_0\wort q$, then $p_1\wort q$. In particular, if $p_0\doms q$ and $p_0\wort q$, then $q$ is realised.
As a consequence, $\wort$ induces a well-defined relation on the domination poset, which we may  expand to $(\invtilde, \doms, \wort)$. By~\cite[Proposition~2.3.31]{invbartheory} the map $\tilde\iota$ from Fact~\ref{fact:fullembemb} is a homomorphism for both $\wort$ and $\nwort$. We prove the analogous statements for orthogonality.
 \begin{pr}\label{pr:perpdoms}
   Let $p_0,p_1,q\in \invtypes(\monster)$. If $p_0\perp q$ and $p_0\doms p_1$, then $p_1\perp q$. In particular, $\perp$ induces a well-defined relation on $\invtilde$.  
 \end{pr}
 \begin{proof}
Fix $r$ witnessing $p_0\doms p_1$ and  let $B\supseteq \monster$. 
    Let $b\models p_1\invext B$ and $c\models q\invext B$. By~\cite[Lemma~1.13]{invbartheory} $(p_0\invext B)\cup r \proves (p_1\invext B)$. Let $a$ be such that $ab\models (p_0\invext B)\cup r$. Since $p_0\perp q$, we have $(p_0\invext B)\wort (q\invext B)$, hence $a\models p_0\invext Bc$. Again by~\cite[Lemma~1.13]{invbartheory} we have $(p_0\invext Bc)\cup r\proves (p_1\invext Bc)$, therefore $b\models p_1\invext Bc$.
  \end{proof}
   \begin{pr}\label{pr:perphom}
In the setting of Fact~\ref{fact:fullembemb},  $\iota\restr\invtypes _{\mathcal A^{<\omega}}(\mathcal A(\monster))$ is  a $\perp$-homomorphism and a $\centernot \perp$-homomorphism, and so is the induced map $\tilde \iota\from \invtildeof{\operatorname{\mathcal A}(\monster)}\into \invtilde$.
\end{pr}
\begin{proof}
Let $p,q\in \invtypes_{\mathcal A^{<\omega}}(\mathcal A(\monster))$ be orthogonal and let $\monster_1\succ \monster$ be $\abs{\monster}^+$-saturated and $\abs{\monster}^+$-strongly homogeneous. We show that, for $\phi(x,y,z)\in L(\monster)$ and $d\in \monster_1$, if  $(\iota p (x)\otimes \iota q(y))\invext \monster_1\proves \phi(x,y,d)$ then 
 $(\iota p\invext \monster d)(x)\cup (\iota q\invext \monster d)(y)\proves \phi(x,y,d)$. By full embeddedness, there are $\chi(x,y,w)\in L_{\mathcal A}(\operatorname{\mathcal A}(\monster))$ and $e\in \operatorname{\mathcal A}(\monster_1)$ such that $\monster_1\models \forall x,y\; (\chi(x,y,e)\coimplica \phi(x,y,d))$. 
Because  $(p\invext \operatorname{\mathcal A}(\monster) e) \wort (q\invext \operatorname{\mathcal A}(\monster) e)$,  there are $\theta_p(x,w), \theta_q(y,w)\in L_{\mathcal A}(\operatorname{\mathcal A}(\monster))$ such that $(p\invext \operatorname{\mathcal A}(\monster) e)\proves \theta_p(x,e)$, $(q\invext \operatorname{\mathcal A}(\monster) e)\proves \theta_q(y,e)$, and $\operatorname{\mathcal A}(\monster_1)\models \forall x,y\; ((\theta_p(x,e)\land \theta_q(y,e))\implica \chi(x,y,e))$. By invariance of $p, q$,
\begin{gather*}
  \pi_p(x)\coloneqq\set{\theta_p(x, e')\mid e'\in \monster_1, e\equiv_{\monster d} e'}\subseteq \iota p\invext \monster_1\\
  \pi_q(y)\coloneqq\set{\theta_q(y, e')\mid e'\in \monster_1, e\equiv_{\monster d} e'}\subseteq \iota q\invext \monster_1
\end{gather*}
So $\pi_p$, $\pi_q$ are consistent. As $\aut(\monster_1/\monster d)$ fixes them, they are equivalent to partial types $\sigma_p$, $\sigma_q$ over $\monster d$. 
But  $\sigma_p\subseteq \iota p\invext \monster d$, $\sigma_q\subseteq \iota q\invext \monster d$, and $\sigma_p(x)\cup \sigma_q(y)\proves \phi(x,y,d)$, proving that $\perp$ is preserved.

Suppose  there is $B$ with $\mathcal A(\monster)\subseteq B\subseteq \mathcal A(\monster_1)$ such that $(p\invext B)\nwort (q\invext B)$. By Remark~\ref{Rem:invimply}, this yields $(\iota p\invext \monster B)\nwort (\iota q\invext \monster B)$, proving that $\centernot \perp$ is preserved as well.

The statement for $\tilde\iota$ follows from Proposition~\ref{pr:perpdoms}.
\end{proof}

\begin{lemma}\label{lemma:opd}
  Let $p,q_0,q_1\in S(\monster)$, with  $p\wort q_0$ and  $(p(x)\cup q_0(y))\doms q_1(z)$, witnessed by $r\in S_{p\otimes q_0, q_1}(M)$. If $(r\restr x)\wort (r\restr yz)$,  then $q_0\doms q_1$, witnessed by $r\restr yz$. Hence, if $\mathcal{A}$, $\mathcal{C}$ are orthogonal families of sorts, $p\in \invtypes_{\mathcal{A}^{<\omega}}(\monster)$, and $q_0, q_1\in \invtypes_{\mathcal{C}^{<\omega}}(\monster)$, if $(p\cup q_0)\doms q_1$, then $q_0\doms q_1$.
  \end{lemma}
\begin{proof}Routine, left to the reader.
\end{proof}
Recall that the product $\prod_{i\in I}P_i$ of a family of posets $(P_i, \le_i)_{i\in I}$  is the cartesian product of the $P_i$ partially ordered by $(p_i)_{i\in I}\le (q_i)_{i\in I}$ iff $\forall i\in I\; p_i\le_i q_i$.
\begin{co}\label{co:orthsortsiso}
  Suppose that $\mathcal{A}$, $\mathcal{C}$ are orthogonal, fully embedded families of sorts. 
  Assume that for every  $p\in \invtypes(\monster)$ there are some $p_\mathcal{A}\in \invtypes_{\mathcal{A^{<\omega}}}(\monster)$ and $p_\mathcal{C}\in \invtypes_{\mathcal{C^{<\omega}}}(\monster)$ such that $p\domeq p_\mathcal{A}\cup p_\mathcal{C}$. Then the map $\class p\mapsto(\class {p_\mathcal{A}}, \class{p_\mathcal{C}})$ is an isomorphism of posets $\invtilde\to \invtildeof{\operatorname{\mathcal A}(\monster)}\times \invtildeof{\operatorname{\mathcal C}(\monster)}$. Moreover, if $\otimes$ respects $\doms$ in $\monster$ (hence also in $\operatorname{\mathcal A}(\monster)$, $\operatorname{\mathcal C}(\monster)$), then this is also an isomorphism of monoids.
\end{co}
\begin{proof}
  Fact~\ref{fact:fullembemb} yields embeddings of posets $\invtildeof{\operatorname{\mathcal A}(\monster)}\into \invtilde$ and $\invtildeof{\operatorname{\mathcal C}(\monster)}\into \invtilde$, yielding a morphism of posets  $\invtildeof{\operatorname{\mathcal A}(\monster)}\times \invtildeof{\operatorname{\mathcal C}(\monster)}\into\invtilde$, which is injective by orthogonality. It is therefore enough to show that the natural candidate for its inverse,  $\class p\mapsto(\class {p_\mathcal{A}}, \class{p_\mathcal{C}})$, is well-defined and a morphism of posets.
Both these statements follow from the observation that, if $(p_\mathcal{A}\cup p_\mathcal{C})\domeq p\doms  q\domeq (q_\mathcal{A}\cup q_\mathcal{C})$, then by Lemma~\ref{lemma:opd} we must have $p_\mathcal{A}\doms q_\mathcal{A}$ and $p_\mathcal{C}\doms q_\mathcal{C}$ The ``moreover'' part follows from Fact~\ref{fact:fullembemb}, and the fact that $\mathcal{AC}$ is fully embedded.
\end{proof}
 \begin{eg}\label{eg:disjunion}
   Let $A$, $C$ be structures in disjoint languages,  $T$ the theory of their disjoint union, in families of sorts $\mathcal{A},\mathcal{C}$. Then $\mathcal{A}$ and $\mathcal{C}$ are orthogonal, and invariant types in $\mathcal A$ are orthogonal to those in $\mathcal C$. So,  $\invtilde$ is isomorphic to $\invtildeof{\operatorname{\mathcal A}(\monster)}\times \invtildeof{\operatorname{\mathcal C}(\monster)}$, and is well-defined as a monoid if and only if both factors are.  
 \end{eg}
  Orthogonality is preserved by the Morley product. The proof is folklore, and essentially the same as in the stable case, but we record it here for convenience.
 \begin{pr}\label{pr:otimesperp}
If $p_0,p_1\in \invtypes(\monster)$ are orthogonal to $q$, then so is $p_0\otimes p_1$. 
 \end{pr}
 \begin{proof}
   Let $ab\models p_0\otimes p_1$ and $c\models q$. Because $p_1\perp q$ we have $c\models q\invext \monster b$, and by definition of $\otimes$ we have $a\models p_0\invext \monster b$. Since $p_0\perp q$, this entails $c\models q\invext \monster ab$.
 \end{proof}

 \section{Regular ordered abelian groups}\label{sec:ro}
 In this section we study the domination monoid in certain theories of (linearly) ordered abelian groups, henceforth \emph{oags}. Model-theoretically, the simplest oags are the (nontrivial) divisible ones. Their theory is o-minimal and  their domination monoid was one of the first ones to be computed~\cite{hhm, dominomin}. It is isomorphic to the finite powerset semilattice $(\pfin(\icsg(\monster)), \cup,  \subseteq)$,  of the set of invariant convex subgroups of $\monster$, and weakly orthogonal classes of types correspond to disjoint finite sets. 
Divisible oags  eliminate quantifiers in the language $L_\mathrm{oag}\coloneqq\set{+,0,-,<}$. In this section we compute the domination monoid in the next simplest case.
 
\begin{defin}
A (nontrivial) oag is \emph{discrete} iff it has a minimum positive element, and \emph{dense} otherwise. We view an oag $M$ as a structure in the \emph{Presburger language} 
     $L_\mathrm{Pres}\coloneqq \set{+,0,-,<,1,\equiv_n\mid
       n\in\omega}$ by interpreting $+,0,-,<$ in the
     natural way, $1$ as the minimum positive element if $M$ is discrete and as $0$ otherwise, and $\equiv_n$ as congruence modulo $nM$.
An oag is \emph{regular} iff it eliminates quantifiers in $L_\mathrm{Pres}$.
 \end{defin}
 
 \begin{fact}[{\cite{robzak, zakon, conrad, weispfenning, choagqe}}]\label{fact:regoag}
   For an oag $M$, the following are equivalent.
   \begin{enumerate}
   \item $M$ is regular.
   \item The only definable convex subgroups of $M$ are $\set 0$ and $M$.
   \item The theory of $M$ has an archimedean model.
   \item \label{point:intdiv} For every $n>1$, if the interval $[a,b]$ contains at least $n$ elements, then it contains an element divisible by $n$.
   \item Every quotient of $M$ by a nontrivial convex subgroup is divisible.
   \end{enumerate}
 \end{fact}
 \begin{fact}[{\cite{robzak, zakon}}]\label{fact:complregoag}
Every discrete regular $M$ is a model of \emph{Presburger Arithmetic}, i.e.\ $M\equiv \mathbb Z$. If $M, N$ are dense regular, then $M\equiv N$ if and only if, for each $\primo\in \mathbb P$, either $M/\primo M$ and $N/\primo N$ are  both infinite or they have the same finite size.
 \end{fact}

\begin{notation}For the rest of the section we adopt the following (not entirely standard) conventions. Let $M$ be an oag and $A\subseteq M$. 
We denote by $A_{>0}$ the set $\set{a\in A\mid a>0}$, by $\Span A$ the group generated by $A$, and by $\operatorname{div}(M)$  the divisible hull of $M$.
We allow intervals to have endpoints in the divisible hull. In other words, an \emph{interval} in $M$ is a set of the form $\set{x\in M\mid a \sqsubset_0 x \sqsubset_1 b}$, for suitable $a,b\in \div(M)\cup\set{\pm \infty}$ and $\set{\sqsubset_0, \sqsubset_1}\subseteq \set{<, \le}$.

A \emph{cut} $(L,R)$ is given by subsets $L, R\subseteq M$ such that $L\le R$ and $L\cup R=M$. We call such a cut \emph{realised} iff $L\cap R\ne \emptyset$, and \emph{nonrealised} otherwise. The \emph{cut} $(L,R)$ of $p\in S_1(M)$ is given by $L=\set{m\in M\mid p(x)\proves x\ge m}$ and $R=\set{m\in M\mid p(x)\proves x\le m}$. The \emph{cut} of $c\in N\succ M$ in $M$ is the cut of $\tp(c/M)$. We say that $c\in N\succ M$ \emph{fills} a cut $(L,R)$ if the latter equals the cut of $c$. For $a\in M$, we denote by $a^+$ the cut $(L,R)$ with $L=\set{m\in M\mid m\le a}$ and $R=\set{m\in M\mid a<m}$, and similarly for $a^-$. Analogous notions are defined for $a\in \div(M)$. 
\end{notation}

Every interval is definable: e.g., $(a/n, +\infty)$ is defined by $a<n\cdot x$. If $(L,R)$ is a cut then $\abs{L\cap R}\le 1$. A type is realised if and only if its cut is. Let $L_\mathrm{ab}\coloneqq\set{0,+,-}$.

\begin{rem}\label{rem:cutsandtypes} By  regularity, a $1$-type over $M\models T$ is determined by a cut in $M$ and a choice of cosets modulo each $nM$ (if $M/nM$ is infinite a type may say that the coset $x+nM$ is not represented in $M$) consistent with the $L_\mathrm{ab}$-theory of $M$.
\end{rem}
\begin{lemma}\label{lemma:anycosets2}
If $M$ is a dense regular oag then, for every $n>0$, every coset of $nM$ is dense in $M$. In particular, given any nonrealised $p\in S_1(M)$, and any nonrealised $q_0\in S_1(M\restr L_{\mathrm{ab}})$, there is $q\in S_1(M)$ restricting to $q_0$ and in the same cut as $p$.
\end{lemma}
\begin{proof}
 By density and point~\ref{point:intdiv} of Fact~\ref{fact:regoag}, every $nM$ is dense; as translations are homeomorphisms for the order topology, each coset of $nM$ is dense.
\end{proof}

\subsection{Imaginaries in regular ordered abelian groups}
The first step to compute $\invtildeof{\monster^\eq}$ is to take care of the reduct to a certain fully embedded family of imaginary sorts, that suffice for weak elimination of imaginaries by a result of Vicar\'ia~\cite{vicEI}.
Recall that $T$ has \emph{weak elimination of imaginaries} iff for every imaginary $e$ there is a real tuple $a$ such that $e\in \dcl^\eq(a)$ and $a\in \acl^\eq(e)$.
For $\primo\in \mathbb N$  and $n\geq1$, define $T_{\primo^n}$ as the $L_\mathrm{ab}$-theory of  $\bigoplus_{i\in\omega}\mathbb{Z}/\primo^n\mathbb{Z}$. The following is well known.

\begin{fact}\label{fact:ladder}
\begin{enumerate}
\item Let $A$ be an infinite abelian group. Then  $A\models T_{\primo^n}$ if and only if $\primo A=\set{a\in A\mid \primo^{n-1}a=0}$.
\item $T_{\primo^n}$ has quantifier elimination and is totally categorical.
\item If $A\models T_{\primo^n}$, then $\primo A$ is a model of $T_{\primo^{n-1}}$, and the induced structure on $\primo A$ is that of a pure abelian group.
\item $T_{\primo^n}$ has weak elimination of imaginaries.\label{WEI-Tpn}
\end{enumerate}
\end{fact}

\begin{proof}[Proof sketch]
For~\ref{WEI-Tpn}, as $T_{\primo^n}$ is stable, it suffices to show that  canonical bases of types over models are interdefinable with real tuples~\cite[Proposition~3]{EvPiPo90}. This is an application of the Elementary Divisor Theorem, and is left to the reader.
  \end{proof}
Let $T_{\primo^{\infty}}$ be the following multi-sorted theory:   
\begin{itemize}
\item for every $n>0$ there is a sort $\sort_{\primo^{n}}$, endowed with a copy of $L_\mathrm{ab}$;  
\item for every $n>0$ there is a function symbol $\rho_{\primo^{n+1}}\from \sort_{\primo^{n+1}}\rightarrow \sort_{\primo^{n}}$;
\item $M\models T_{\primo^{\infty}}$ if and only if, for all $n>0$,  $\sort_{\primo^n}(M)\models T_{\primo^n}$ and $\rho_{\primo^{n+1}}\from \sort_{\primo^{n+1}}(M)\to \sort_{\primo^{n}}(M)$ is a surjective group 
homomorphism with kernel $\primo^n\sort_{\primo^{n+1}}(M)$.
\end{itemize}

\begin{rem}\label{rem:gehretqe}
In an earlier version of this manuscript, we had claimed that $T_{\primo^{\infty}}$ has quantifier elimination. This does not hold. But one may show that it is enough to add function symbols $\lambda_n\from\sort_{\primo^{n}}\rightarrow\sort_{\primo^{n+1}}$ for all $n$, interpreted as the definable group isomorphism $\sort_{\primo^{n}}(A)\rightarrow \primo\sort_{\primo^{n+1}}(A)$ mapping $a$ to $\primo\tilde{a}$ where $\tilde{a}$ is any element with $\rho_{\primo^{n+1}}(\tilde{a})=a$. We thank Allen Gehret for having pointed this out to us.
\end{rem}
The quantifier elimination result above, which has been mentioned for the sake of completeness, will not be used below. Let  $\hat \kappa$ be the monoid of  cardinals not larger than $\kappa$, with the usual  sum and order.
\begin{co}\label{co:ladder}
  \begin{enumerate}
  \item\label{point:tpinfty1}   The theory $T_{\primo^{\infty}}$ is complete, totally categorical, 1-based, and has weak elimination of imaginaries.
  \item\label{point:tpinfty2}  In $T_{\primo^\infty}$, we have $\invtilde\cong \mathbb N$ and, for each infinite cardinal $\kappa$, the monoid $\invtildestarof\kappa\monster$ is (well-defined and) isomorphic to $\hat \kappa$.
  \item  \label{point:tpinfty3} More precisely, if $\tp(a/\monster)$ is $M$-invariant, then there is a basis  $b\in \dcl(M a)$ of the $\mathbb F_\primo$-vector space $\sort_\primo(\dcl(\monster a))$ over $\monster$, and $\tp(a/\monster)$ is domination-equivalent to $\tp(b/\monster)$, witnessed by $\tp(ab/M)$ in both directions, and  the isomorphism above sends its domination-equivalence class to the cardinality of $b$.  \end{enumerate}
\end{co}

\begin{proof}
\ref{point:tpinfty1} is immediate from  Fact~\ref{fact:ladder} and the fact that abelian groups are 1-based.
As for~\ref{point:tpinfty2}, each of the sorts $\sort_{\primo^n}$ is stable unidimensional, that is, if $p\perp q$ then one of $p,q$ is algebraic, and it follows easily that so is $T_{\primo^\infty}$. The conclusion for finitary types then follows from~\cite[Corollary~5.19]{invbartheory}, and the version for \textasteriskcentered-types is similar.

To prove~\ref{point:tpinfty3}, if $b\in \dcl(Ma)$ is a basis of $\sort_\primo(\dcl(M a))$ over $M$,  by $M$-invariance it is also a basis of $\sort_\primo(\dcl(\monster a))$ over $\monster$. Because in unidimensional theories the domination-equivalence class of a tuple is determined by its weight~\cite[Remark~5.12]{invbartheory}, it suffices to show that the cardinality $\kappa$ of $b$ equals the weight  $w(\tp(a/\monster))$. For $T_{\primo^n}$ this is well known, and as $\sort_{\primo^n}(\monster)$ is a fully embedded model of $T_{\primo^n}$, the result is easily seen to transfer to $T_{\primo^{\infty}}$.
\end{proof}

We now consider a regular oag $M$. Since it is well known that Presburger Arithmetic   eliminates imaginaries (by definable choice), we may assume that $M$ is dense.

We view $M$ as a structure in the language with one sort for the oag itself, endowed with $L_\mathrm{oag}$, one sort $\sort_{\primo^n}$ for each prime $\primo$ and each $n>0$, endowed with $L_\mathrm{ab}$ and interpreted as the group $M/\primo^nM$, functions $\pi_{\primo^n}$ for the quotient map from $M$ to $M/\primo^nM$ and functions $\rho_{\primo^{n+1}}$ for the canonical surjections $M/\primo^{n+1}M\rightarrow M/\primo^nM$. Moreover, for every prime $\primo$ we definably expand the  language on $\left(\sort_{\primo^n}\right)_{n>0}$ so that the multi-sorted structure  $\left(\sort_{\primo^n}(M)\right)_{n>0}$ has quantifier elimination. 

For every $\primo\in \mathbb P$, let $d_\primo\in\mathbb{N}\cup\set{\infty}$ be such that $(M: \primo M)=\primo^{d_\primo}$. Set $T\coloneqq \Th(M)$.  The proof of the following lemma  is straightforward from Lemma~\ref{lemma:anycosets2} and quantifier elimination for the one-sorted theory of $M$ in $L_\mathrm{Pres}$, and we leave it to the reader.

\begin{lemma}\label{lemma:roag-multisorted}
The theory $T$ eliminates quantifiers. For $\monster\models T$, the following holds.
 For every $\primo$ prime and $n>0$, the sort $\sort_{\primo^n}(\monster)$ equipped with the natural $L_\mathrm{ab}$-structure is fully embedded.
 If $d_\primo=\infty$, the structure given by $(\sort_{\primo^n}(\monster))_{n>0}$, together with the maps $\rho_{\primo^{n+1}}$ and the natural $L_\mathrm{ab}$-structure on each sort, is fully embedded and a model of $T_{\primo^{\infty}}$. If $d_\primo$ is finite, every sort $\sort_{\primo^n}(\monster)$ is finite.
 If $\mathfrak p, \mathfrak q$ are distinct primes, then $\sort_{\primo^n}(\monster)_{n>0}$ and $\sort_{\mathfrak q^n}(\monster)_{n>0}$ are orthogonal.
\end{lemma}
\begin{defin}
 Denote by  $\mathcal Q$ the family of sorts $\set{\sort_{\primo^n}\mid \primo\in \mathbb P, n>0}$. If $q=\tp(c/\monster)$ is a \textasteriskcentered-type, possibly with coordinates in the sorts in $\mathcal Q$, for each  $\primo\in \mathbb P$,  let $\kappa_\primo(q)$ be the dimension of the $\mathbb F_\primo$-vector space $\dcl(\monster c)/\primo(\dcl(\monster c))$ over $\monster/\primo \monster$.  Let $\mathbb P_T$ be the set of primes $\primo$ such that if $M\models T$ then $\primo M$ has infinite index, and denote by  $\prod_{\mathbb P_T}\hat \kappa$ the monoid of $\mathbb P_T$-indexed sequences of cardinals smaller or equal than $\kappa$ with pointwise cardinal sum, equipped with the  product (partial) order.
\end{defin}

\begin{co}\label{co:invQ}
  The family of sorts $\mathcal Q$, equipped with the $L_\mathrm{ab}$-structure on each sort and the maps $\rho_{\primo^{n+1}}$, is fully embedded. When viewed as a standalone structure, $\otimes$ respects $\doms$ and $\invtildestarof\kappa{\mathcal Q(\monster)}\cong \prod_{\mathbb P_T}\hat \kappa$.
\end{co}
\begin{proof}
This follows from Lemma~\ref{lemma:roag-multisorted}, Corollary~\ref{co:ladder}, and Fact~\ref{fact:fullembemb}. Compatibility of $\otimes$ with $\doms$ is a consequence of stability, see~\cite[Propositions~1.21 and 1.25]{invbartheory}.
\end{proof}
\begin{fact}[{\cite[Theorem~5.1]{vicEI}}]\label{thm:oagswei}
The theory $T$ has weak elimination of imaginaries.
\end{fact}

\begin{rem}\label{rem:ppenough}
In \cite{vicEI}, Vicar\'ia proves a more general result, of which Fact~\ref{thm:oagswei} is a special case. Note that she adds sorts for quotients of the form $M/nM$ for all $n>0$. As $M/nM$ is definably isomorphic to $\prod_{i=1}^mM/\primo_i^{n_i}M$, where $n=\prod_{i=1}^m\primo_i^{n_i}$ is the decomposition of $n$ into prime powers, it suffices to add the sorts $\sort_{\primo^n}$.

 Observe that, for the above to go through, we need to have in our language the sorts $\sort_{\primo^n}$ even when they are finite.  Alternatively, one may dispense with the finite $\sort_{\primo^n}$ by naming enough constants, e.g.~by naming a model.
\end{rem}
\subsection{Moving to the right of a convex subgroup}
 \begin{ass}
     Until the end of the section, $T$ is the complete $L_\mathrm{Pres}$-theory of a regular oag. Imaginary sorts are not in our language until further notice.
 \end{ass}

\begin{defin}\label{defin:invconvsbgrs}Let  $B\subseteq M$.
   A type $q(x)\in S_1(M)$ is \emph{right of $B$}  iff $q(x)\proves\set{x>d\mid d\in B}\cup \set{x<d\mid d\in M, d>B}$.
 An element of an elementary extension of $M$ is \emph{right of $B$} if its type over $M$ is.
 A convex subgroup $H$ of $\monster$ is called \emph{[$A$-]invariant} iff there is an [$A$-]invariant type to its right.
\end{defin}
\begin{rem}\label{rem:smallcof} Let $p\in S_1(\monster)$ be an $M$-invariant type. If its cut $(L,R)$ is definable, then it is $M$-definable. If not, then exactly one between the cofinality of $L$ and the coinitiality of $R$ is small, and $M$ contains a set cofinal in $L$ or coinitial in $R$.
\end{rem}

\begin{proof}
The case of a definable cut is clear, so let us assume $(L,R)$ is a non-definable cut of $\monster$.  In particular, $L\neq\emptyset\neq R$. If $L\cap M$ is not cofinal in $L$, there is $\ell\in L$ with $L\cap M<\ell$, so by regularity of $\monster$ and saturation there is $\ell_0\in L$ divisible by all $n\geq1$ such that $L\cap M<\ell_0<\ell$. Similarly, if $R\cap M$ is not coinitial in $M$ there is $r_0\in R$, which is divisible by all $n\geq1$, such that $r_0<R\cap M$. By Remark~\ref{rem:cutsandtypes} it follows that $\tp(\ell_0/M)=\tp(r_0/M)$, showing that $(L,R)$ is not $M$-invariant.
\end{proof}

 In particular, in a regular oag a nontrivial convex subgroup $H$ of $\monster$ is invariant if and only if the cofinality of $H$ or the coinitiality of $(\monster\setminus H)_{>0}$ is small, while the trivial subgroup $\set{0}$ is invariant if and only if $\monster$ is dense.
\begin{lemma}\label{lemma:movetocsg}
  In the theory of a regular oag, suppose that $p\in \invtypes_1(\monster)$ and $f$ is a definable function such that $f_*p$ is not realised. Then $p\domeq f_*p$.
\end{lemma}
\begin{proof}  
   Clearly $p\doms f_*p$. By~\cite[Cor~1.10]{choagqe}, $f$ is piecewise affine. As $f_*p$ is not realised, $f$ cannot be constant at $p$, so it is invertible at $p$ and $f_*p\doms f_*\inverse(f_*p)=p$.
\end{proof}

\begin{pr}
  In Presburger Arithmetic, every invariant $1$-type is domination-equivalent to a type right of an invariant convex subgroup.
\end{pr}
\begin{proof}
By Lemma~\ref{lemma:movetocsg} it suffices to show that, for every nonrealised  $p\in \invtypes_1(\monster)$ there is a definable  $f$ such that $f_*p$ is right of an invariant convex subgroup. By Fact~\ref{fact:regoag},  $\monster/\mathbb Z$ is divisible, and it is easy to see that $\monster/\mathbb Z$ inherits saturation and strong homogeneity from $\monster$. The conclusion follows by lifting the analogous result~\cite[Corollary~13.11]{hhm} (see also~\cite[Proposition~4.8]{dominomin}) from $\monster/\mathbb Z$.  
\end{proof}
In the rest of the subsection we generalise the above to the regular case.
\begin{ass}
   Until the end of the subsection, $M$  denotes  a dense regular oag, and $\monster$ a monster model of $T\coloneqq \Th(M)$. 
\end{ass}

\begin{pr}\label{pr:coinelext}
  Let $b\in \monster \setminus M$ be divisible by every $n>1$ and let $B\coloneqq \Span{M b}= M+ \mathbb Q b$. If $M_{>0}$ is coinitial in $B_{>0}$, then $M\prec B\prec \monster$.
\end{pr}
\begin{proof}
    The inclusion $M\subseteq B$ is \emph{pure}, i.e.~for every $n>1$ we have $nB\cap M=nM$.
Moreover, if $c=a+\gamma b$, with $a\in M$ and $\gamma\in \mathbb Q$, then for every $n$ we clearly have $c-a\in nB$, hence $B/nB$ may be naturally identified with $M/nM$.

Because $M$ is dense and $M_{>0}$ is coinitial in $B_{>0}$, it follows that $B$ is as well dense. Let $c<d\in B$ and $n>1$. By assumption, $(0,d-c)$ intersects $M$, hence contains an interval $I$ of $M$, hence represents all elements of $M/nM$ by Lemma~\ref{lemma:anycosets2}. These can be identified with the elements of $B/nB$, as observed above, so there is $e\in I$ such that $c+e\in nB$. Clearly, $c+e \in (c,d)$, hence $B$ is regular by Fact~\ref{fact:regoag}.

 By Fact~\ref{fact:complregoag} and the identification of $M/nM$ with $B/nB$, we obtain  $B\equiv M$.  Since  $M$  is pure in $B$, it is an $L_\mathrm{Pres}$-substructure of $B$, and the conclusion follows by quantifier elimination in $L_\mathrm{Pres}$.
\end{proof}
Recall that an extension $A<B$ of oags is an \emph{i-extension} iff there is no $b\in B_{>0}$ such that the set $\set{a\in A\mid a<b}$ is closed under sum. \begin{lemma}\label{lemma:clearden}
Let $H<M<N$, with $M$ dense regular and $H$ convex.  The set of elements of $N$  right of $H$ is closed under sum. In particular, $N$ is an i-extension if and only if $H\mapsto H\cap M$ is a bijection between the convex subgroups of $N$ and $M$.
\end{lemma}
\begin{proof}
  If $H=M$, the statement is trivial. If $H=\set{0}$,   let $0<c,d<M_{>0}$ and pick $a\in M_{>0}$. By density, there is $b\in M$ with $0<b<a$, and since $b$ and $a-b$ are both in $M_{>0}$ we conclude $c+d<b+a-b=a$. If $H$ is proper nontrivial, by Fact~\ref{fact:regoag} the quotient $M/H$ is divisible, and the conclusion follows from the previous case applied to $M/H$ as a subgroup of the quotient of  $N$ by the convex hull of $H$.
\end{proof}

\begin{pr}
Every $M\models T$ has a maximal elementary i-extension.
\end{pr}
\begin{proof}
This is easy, see e.g.~\cite[Proposition~4.2.17]{mennuni_thesis}.
\end{proof}
\begin{pr}\label{pr:csgbij2}
  Suppose $M\models T$ has no proper elementary i-extension and let $p \in S_1(M)$ be nonrealised. Then there are  $a\in M$ and $\beta \in \mathbb Z\setminus\set 0$ such that, if $f(t)=a+\beta t$, then the pushforward $f_*p$ is right of a convex subgroup.
\end{pr}
\begin{proof}
  Let $b\models p$, and suppose first that $b$ is divisible by every $n$. Consider $B\coloneqq \Span{Mb}=M+ \mathbb Q b$. If  there are $a'\in M$ and $\beta'\in \mathbb Q$ such that $0<a'+ \beta' b< M_{>0}$, by Lemma~\ref{lemma:clearden}  multiplying by the denominator of $\beta'$ yields a positive element smaller than $M_{>0}$,  so we obtain the conclusion with the convex subgroup $\set 0$.  If instead there is no such $a'+\beta' b$, then $M_{>0}$ is coinitial in $B_{>0}$, and by Proposition~\ref{pr:coinelext} $B\succ M$. By maximality of $M$, there must be convex subgroups $H_0\subsetneq H_1$ of $B$ such that $H_0\cap M=H_1\cap M$. Hence any positive $a+\beta b\in H_1\setminus H_0$ is right of $H_0\cap M$. We conclude again by clearing the denominator of $\beta$ and using Lemma~\ref{lemma:clearden}.

This shows the conclusion when $b$ is divisible by all $n$.  In the general case, by Lemma~\ref{lemma:anycosets2}, there is $c\in \monster$ with the same cut in $M$ as $b$ which is divisible by every $n$. As we just proved, there is $f(t)\coloneqq a+\beta t$, with  $\beta\in \mathbb Z$ and $a\in M$, such that the cut of  $f(c)$ in $M$ is that of a convex subgroup. Because $f(t)$ sends intervals to intervals, it sends cuts to cuts, hence the cut of $f(b)$ equals that of $f(c)$.
\end{proof}

\begin{co}\label{co:1tpdeqsbgrp}
For  every nonrealised $p(x) \in \invtypes_1(\monster)$ there is a  definable function $f$ such that $(f_*p)(y)$ is right of an invariant convex subgroup, and domination-equivalent to $p$, witnessed by any small type containing $y=f(x)$.
\end{co}
\begin{proof}
  If $p$ is $M$-invariant, up to enlarging $M$ we may assume that it has no proper elementary i-extension. Let $f(t)$ be an $M$-definable function given by Proposition~\ref{pr:csgbij2} applied to $p\restr M$. Then $f_*p$ is $M$-invariant, and its cut is either the one to the left of $(f_*p\restr M)(\monster)$ or the one to its right, which are both cuts right of convex subgroups of $\monster$ by Lemma~\ref{lemma:clearden}. Now apply Lemma~\ref{lemma:movetocsg}.
\end{proof}

\subsection{Computing the domination monoid}
By Fact~\ref{thm:oagswei}, regular oags weakly eliminate imaginaries after adding the sorts $\sort_{\primo^n}$. 
As already  remarked, this implies that passing to $T^\eq$ does not affect the poset $\invtilde$, nor its well-definedness as a monoid. Hence, we will conflate the two settings, and refer to our theory in this language as $T^\eq$,  reserving $T$ for the 1-sorted $L_\mathrm{Pres}$-theory of a regular oag.
\begin{ass}
 Until the end of the section, we work in $T^\eq$.
\end{ass}
\begin{lemma}\label{lemma:wortcsg}
  Let $H_0\subsetneq H_1$ be convex subgroups of $M\models T$ and, for $i<2$, let $q_i(x^i)\in S_1(M)$ be right of $H_i$. Suppose that there is no prime $\primo\in \mathbb P$ such that both $q_i(x^i)$ prove that $x^i$ is in a new coset modulo some $\primo^{\ell_i}$ Then $q_0\wort q_1$.
\end{lemma}
\begin{proof}
 By  Lemma~\ref{lemma:clearden} the cut of every $k_0x^0+k_1x^1$ is determined by $q_0(x^0)\cup q_1(x^1)$, and we conclude by assumption and quantifier elimination.
\end{proof}
\begin{pr}\label{pr:wortto1tp}
  Suppose that $q_H(x)\in\invtypes_1(\monster)$ is right of the convex subgroup $H$ and prescribes realised cosets modulo every $n$ for $x$. For an invariant \textasteriskcentered-type $q$ with all coordinates in the home sort, the following are equivalent.
  \begin{enumerate}
  \item\label{point:qhfun} For every (equivalently, some) $b\models q$, no type right of $H$ is realised in $\Span{\monster b}$.
  \item\label{point:qhwort} $q_H\wort q$.
  \item \label{point:qhcomm} $q_H$ commutes with $q$.
  \item \label{point:qhperp} $q_H\perp q$.
  \end{enumerate}
Moreover, if $q'$ is a \textasteriskcentered-type with no coordinates in the home sort, then $q_H\perp q'$.
\end{pr}
\begin{proof}
  To show  $\ref{point:qhfun}\allora \ref{point:qhwort}$, consider $q_H(x)\cup q(y)$. By assumption on $q_H$ we only need to deal with inequalities of the form $kx+\sum_{i<\abs y} k_i y_i+d\ge 0$, but $\ref{point:qhfun}$ gives immediately that the cut of $kx$ in $\Span{\monster b}$ is determined.   If~\ref{point:qhfun} fails, as witnessed by $f(b)$, say, then $q_H(x)\otimes q(y)$ and $q(y)\otimes q_H(x)$ disagree on the formula $f(y)<x$, proving $\ref{point:qhcomm}\allora \ref{point:qhfun}$, and $\ref{point:qhwort}\allora \ref{point:qhcomm}$ holds for every type in every theory.

  We prove $\ref{point:qhwort}\allora\ref{point:qhperp}$, the converse being trivial. Suppose  that $B\supseteq \monster$ is such that $(q_H\invext B)\nwort (q\invext B)$.  The cosets modulo every $n$ of a realisation of $q_H$ are all realised in $\monster$, so there must be some inequality of the form $kx+\sum_{i<\abs y} k_i y_i+d\ge 0$, with $k_i\in \mathbb Z$ and $d\in \Span B$, that is not decided. Hence, if~\ref{point:qhperp} fails, it fails for a $1$-type $\tilde q$, namely the pushforward of $q$ under the map  $y\mapsto\sum_{i<\abs y} k_i y_i$. By Corollary~\ref{co:1tpdeqsbgrp} and Proposition~\ref{pr:perpdoms}, we may  assume $\tilde q$ is right of a convex subgroup. Therefore $q_H(x)$ and $\tilde q(z)$ are weakly orthogonal by~\ref{point:qhwort} and~\cite[Proposition~3.13]{invbartheory}, to the right of distinct (by weak orthogonality) convex subgroups, but the cut in $\Span B$ of $kx+z$ is not determined by $(q_H\invext B)(x)\cup (\tilde q\invext B)(z)$. This contradicts Lemma~\ref{lemma:clearden}.

  For the ``moreover'' part, by Proposition~\ref{pr:perpdoms} we may replace $q'$ with any domination-equivalent type, so we may assume, using Corollary~\ref{co:invQ} and Proposition~\ref{pr:perphom}, that $q'(z)$ is the type of an independent tuple, with $z_i\in\sort_{\primo_i}$. Let $H'$ be any invariant convex subgroup different from $H$, let $p_i$ be the $1$-type right of $H'$ in a new coset modulo $\primo_i$ and congruent to $0$ modulo every other prime, and let $q$ be the tensor product, in any order, of the $p_i$. Clearly $q\doms q'$ and, by construction, if $b\models q$ then no type right of $H$ is realised in $\Span{\monster b}$, so we conclude by Proposition~\ref{pr:perpdoms}.
\end{proof}

\begin{defin}
  Let $q$ be an invariant global \textasteriskcentered-type, and  $c\models q$.  Let $\mathcal H(q)$ be the set of cuts of convex subgroups of $\monster$ filled in $\Span{\monster c}$.
\end{defin}

\begin{thm}\label{thm:regoaddec1tp}
If $p,q$ are invariant \textasteriskcentered-types, then $p\doms q$ if and only if
 $\mathcal H(p)\supseteq \mathcal H(q)$ and  $\forall\primo\in \mathbb P\;\kappa_\primo(p)\ge\kappa_\primo(q)$.
Hence, $\class q$ is determined by $\mathcal H(q)$ and  $\primo\mapsto \kappa_\primo(q)$.
\end{thm}
\begin{proof}  Let $c\models q$, and write $c=c^0c^1$, with $c^0$ a tuple in the home sort and $c^1$ a tuple from the sorts $\sort_{\primo^n}$. By enlarging $c_1$ with at most $\abs c$ points of $\dcl(\monster c)$ if necessary, we may assume that it contains bases of all $\mathbb F_\primo$-vector spaces $\sort_\primo(\dcl(\monster c))$ over $\monster$. Observe that this is harmless domination-wise, and that it does not impact compatibility of $\otimes$ with $\doms$  by \cite[Proposition~1.23]{invbartheory}.

  Index on a suitable cardinal $\kappa$, bounded by the cardinality of $c^0$, the (necessarily invariant) convex subgroups $H_j$ whose cuts are filled in $\Span{\monster c^0}$. Note that, by Corollary~\ref{co:1tpdeqsbgrp}, we have  $\kappa\ne 0$ unless  $c^0$ is realised.

  For $j<\kappa$, let $q_j(y_j)$ be the type right of $H_j$ divisible by every nonzero integer. By Lemma~\ref{lemma:wortcsg} and Proposition~\ref{pr:wortto1tp}, the $q_j$ are orthogonal, and it follows from Proposition~\ref{pr:otimesperp} and compactness that their union is a complete type; call it $q_{\mathrm{H}}(y)$. Let $q_{\mathrm{Q}}(z)\coloneqq\tp(c_1/\monster)$.  By Proposition~\ref{pr:wortto1tp} and Proposition~\ref{pr:otimesperp} $q_{\mathrm{H}}\perp q_{\mathrm{Q}}$.  We prove that $q(x)$ is domination-equivalent to $q'(yz)\coloneqq q_{\mathrm{H}}(y)\otimes q_{\mathrm{Q}}(z)$. If $c^0\in \monster$, equivalently if $q_{\mathrm{H}}$ is realised, this is trivial, so we assume this is not the case.

To show $q'(yz)\doms q(x)$, let $b\in \dcl(\monster c)$ be maximal amongst the tuples with
 each $b_k$ in the cut of an invariant convex subgroup, and such that  if $k<k'$ then  $\Span{b_k}_{>0}<\Span{b_{k'}}_{>0}$. A maximal such $b$ exists because the size of $b$ is at most that of $c^0$,  by looking at $\mathbb Q$-linear dimension over $\monster$ in the divisible hull.  Since $c^0\notin \monster$, by Corollary~\ref{co:1tpdeqsbgrp}  there is a point of $\dcl(\monster c)$ in the cut of an invariant  convex subgroup, hence $b$ is nonempty. By \cite[Corollary~1.10]{choagqe} definable functions are piecewise affine and, by clearing denominators using Lemma~\ref{lemma:clearden}, we may assume that $b\in \Span{\monster c^0}$. 

Write $b_k=f_k(c^0)$, for suitable affine functions $f_k$. Let $M\smallprec \monster$ be large enough to contain the parameters of the $f_k$,  such that $q$ and $q'$ are $M$-invariant, and such that $M$ has no proper elementary i-extension. Let $r\in S_{qq'}(M)$ contain the following.
\begin{enumerate}
\item For each $k$, by choice of $q'$ there is $j<\kappa$ such that $y_j$ is in the same cut as $b_k$ according to $q'$. If the cut of $b_k$ has small cofinality on the right, put in $r$ the formula $f_k(x)>y_j$; if it has small cofinality on the left, put in $r$ the formula $f_k(x)<y_j$.
\item \label{point:gluecosets} For each $j<\abs{c^1}$, the formula $x_{\abs{c^0}+j}=z_{j}$.
\end{enumerate}
By Lemma~\ref{lemma:roag-multisorted},  point~\ref{point:gluecosets} above, the fact that $c^1$ contains bases of all $\mathbb F_\primo$-vector spaces $\sort_\primo(\dcl(\monster c))$ over $\monster$, and Corollary~\ref{co:ladder},  to prove  $q'\doms q$ it suffices to show that $q'\cup r$ decides the cut in $\monster$ of every  $\sum_i \delta_ix_i$. We first prove a special case.
\begin{claim}
  $q'\cup r$ entails the quantifier-free $\set{+,0,-,<}$-type of the $f_k(x)$ over $\monster$.
\end{claim}
\begin{claimproof}
It is enough to show that the cut of every $\sum_k \beta_k f_k(x)$ in $\monster$ is decided, where only finitely many $\beta_k\in \mathbb Z$ are nonzero. By choice of $r$ and Remark~\ref{rem:smallcof}, $q'\cup r$ determines the cut of each $f_k(x)$ over $\monster$. Moreover, $r$ contains the information that $\Span{f_k(x)}_{>0}<\Span{f_{k'}(x)}_{>0}$ for $k<k'$. By this, the fact that the $f_k(x)$ are right of convex subgroups, and Lemma~\ref{lemma:clearden},  the cut of $\sum_k \beta_k f_k(x)$ must be that of $\operatorname{sign}(\beta_k)f_k(x)$, with $k$  the largest such that $\beta_k\ne 0$.
\end{claimproof}
As $M$ has no proper elementary i-extension,  given a term $\sum_i \delta_ix_i$, by Proposition~\ref{pr:csgbij2} we can compose with an $M$-definable injective affine function and reduce to a term $\sum_i \gamma_ix_i+d$, with $d\in M$ and $\gamma_i\in \mathbb Z$, with cut in $M$ right of a convex subgroup. As $\tp(\sum_i  \gamma_ix_i+d/\monster)$ is $M$-invariant, $\sum_i \gamma_ix_i+d$ is in the cut of an $M$-invariant convex subgroup of $\monster$.  By maximality of $b$, there must be $k$ and  positive integers $n,m$ such that   $n b_k\le m\left(\sum_i \gamma_ix_i+d\right) \le (n+1) b_k$. Thus $r\proves  n f_k(x)\le m\left(\sum_i \gamma_ix_i+d\right) \le (n+1) f_k(x)$, 
and by the Claim $q'\doms q$.

Similar arguments show $q\doms q'$ and that, if $\mathcal H(p)\supseteq \mathcal H(q)$ and $\forall\primo\in \mathbb P\;\kappa_\primo(p)\ge\kappa_\primo(q)$, and $p'$ is defined analogously to $q'$, then $p'\doms q'$.  That $\mathcal H(p)\supseteq \mathcal H(q)$ is necessary to have $p\doms q$ follows from Proposition~\ref{pr:wortto1tp} and \cite[Proposition~3.13]{invbartheory}. As for $\forall\primo\in \mathbb P\;\kappa_\primo(p)\ge\kappa_\primo(q)$, if for some $\primo \in \mathbb P$ we have $\kappa_\primo(q)>\kappa_\primo(p)$ then we easily find a type in the quotient sorts  dominated by $q$ but not by $p$, a contradiction.
\end{proof}

\begin{pr}\label{pr:regisomon}
For all invariant \textasteriskcentered-types $p,q$ and $\primo \in \mathbb P$, we have $\mathcal H(p\otimes q)= \mathcal H(p)\cup \mathcal H(q)$ and $\kappa_\primo(p\otimes q)=\kappa_\primo(p)+\kappa_\primo(q)$.
\end{pr}
\begin{proof}
  By Proposition~\ref{pr:wortto1tp} $\mathcal H(q)$ is precisely the set of convex invariant subgroups $H$ such that $q\centernot\perp q_H$. By Proposition~\ref{pr:otimesperp}, we therefore have the first statement. The second one is an easy consequence of the definition of $\otimes$.
\end{proof}

Note that if $q\in \invtypes_{<\kappa^+}(\monster)$ then $\abs{\mathcal H(q)}$ and each $\kappa_\primo(q)$ are at most $\kappa$.
\begin{defin}
We denote by $\icsg(\monster)$  the set of invariant convex subgroups of $\monster$, and by $\mathscr P_{\le\kappa}(\icsg(\monster))$ the  monoid of its subsets of size at most $\kappa$ with union, partially ordered by inclusion.
\end{defin}

\begin{co}[{Theorem~\ref{athm:oagsdom}}]\label{co:isro}
For $T$  the theory of a regular oag and $\kappa$ a small infinite cardinal,  $\invtildestarof\kappa{\monster^\eq}$ is well-defined, and $\invtildestarof\kappa{\monster^\eq}\cong \mathscr P_{\le\kappa}(\icsg(\monster))\times \prod_{\mathbb P_T} \hat \kappa$.
  \end{co}
  \begin{proof}
    Compatibility of $\otimes$ and $\doms$  follows from Theorem~\ref{thm:regoaddec1tp} and  Proposition~\ref{pr:regisomon}.
The same results show that the map $\class p$ to $(\mathcal H(p),\primo\mapsto \kappa_\primo(p))$ is well-defined, an embedding of posets, and a morphism of monoids.  Surjectivity is easily checked.
  \end{proof}
In general, the embedding  $\invtildestarof\kappa\monster\into \invtildestarof\kappa{\monster^\eq}$ is not surjective, although its image may be easily computed.  We state the result of this computation, which we leave to the reader, and of the analogous ones for finitary types. Denote by $\prod^\mathrm{bdd}_{\mathbb P_T}\omega$ the  submonoid of $\prod_{\mathbb P_T}\hat\omega$ consisting of bounded sequences of natural numbers.
\begin{co}\label{co:itro}
  The monoids  $\invtildestarof\kappa\monster,\invtilde,\invtildeof{\monster^\eq}$ are all well-defined, and
  \begin{align*}
    \invtildestarof\kappa\monster\cong& \biggl(\mathscr P_{\le\kappa}(\icsg(\monster))\times \prod_{\mathbb P_T} \hat \kappa\biggr)\setminus \set{(a,b)\mid a=\emptyset, b\ne 0}\\
    \invtildeof{\monster^\eq}\cong& \biggl(\mathscr P_{<\omega}(\icsg(\monster))\times \prod^\mathrm{bdd}_{\mathbb P_T} \omega\biggr)\setminus  \set{(a,b)\mid a=\emptyset,
      \operatorname{supp}(b)\textnormal{ infinite}}\\
        \invtilde\cong& \biggl(\mathscr P_{<\omega}(\icsg(\monster))\times \prod^\mathrm{bdd}_{\mathbb P_T} \omega\biggr)\setminus  \set{(a,b)\mid a=\emptyset, b\ne0 }
  \end{align*}
\end{co}

\section{Pure short exact sequences}\label{sec:ses}
 We study pure short exact sequences of abelian structures  $0\to \mathcal A\xrightarrow{\iota}\mathcal  B \xrightarrow{\nu} \mathcal C\to 0$, where $\mathcal A$ and $\mathcal C$ may be equipped with extra structure. 
We view them as multi-sorted structures, and  use the relative quantifier elimination results from~\cite{acgz20} to describe the domination poset in terms of $\mathcal{A}$ and $\mathcal{C}$.  A decomposition of the form $\invtildeof{\operatorname{\mathcal A}(\monster)}\times \invtildeof{\operatorname{\mathcal C}(\monster)}$ only holds in special cases; in general we will need to look at \textasteriskcentered-types and introduce a family of imaginaries of $\mathcal{A}$ which depends on $\mathcal{B}$.

We refer the reader to~\cite[Subsection~4.5]{acgz20} for definitions. We adopt almost identical notations, with the following differences.  We write $\mathcal A$ for an abelian structure and $L$ for its language.
We denote by $\mathcal F$ a fundamental family of pp formulas for $\mathcal B$. The corresponding family of quotient sorts of $\mathcal A$ is denoted by $\mathcal A_{\mathcal F}$. An \emph{$\mathcal{A}$-sort} is simply a sort in $\mathcal A$. We write e.g.~$t(x)$ for a tuple of terms, $0$ for a tuple of zeroes of the appropriate length, etc. Tuples of the same length may be added, and tuples of appropriate lengths used as arguments, as in  $f(t(x,0)-d)=0$. 
\begin{eg}\label{eg:abstr}
 In the simplest abelian structures, namely  abelian groups, $\mathcal F\coloneqq\set{\exists y\; x=n\cdot y\mid n\in \omega}$ is always fundamental. In an arbitrary abelian structure, one may always resort to taking as $\mathcal F$ the trivially fundamental set of all pp formulas.
\end{eg}
\begin{rem}\label{rem:termhom}
In an $L$-abelian structure, each $L$-term $t(x)$ is built from homomorphisms  $f_j$ of abelian groups by taking $\mathbb Z$-linear combinations and compositions. Hence, $t(x)$ is itself a homomorphism of abelian groups.
\end{rem}

A short exact sequence of abelian groups $0\to \mathrm{A}\to \mathrm{B}\to \mathrm{C}\to 0$ is pure if and only if, for each $n$, we have $n\mathrm{B}\cap \mathrm{A}= n\mathrm{A}$. This holds, e.g.,  if $\mathrm{C}$ is torsion-free, and in particular in the two examples below. We may take as $\mathcal F$ that of Example~\ref{eg:abstr}.

\begin{eg}\label{eg:ses1}
Suppose that the expansion $L_\mathrm{ac}^*$ endows $\mathrm{A}$, $\mathrm{C}$ with the structure of oags. Note that one then recovers, definably, an oag structure on $\mathrm{B}$, induced by declaring that $\iota(\mathrm{A})$ is convex. Because of this, and of fact that the kernel of a morphism of oags is convex, this setting is equivalent to that of a short exact sequence of oags.
This will be used in Section~\ref{sec:owfmig}, with $\mathrm{B}$ an oag  and $\mathrm{A}$ a suitably chosen convex subgroup. The sorts $\mathrm{A}_\phi$ coincide with the quotients $\mathrm{A}/n\mathrm{A}$.
\end{eg}

\begin{eg}\label{eg:ses2}
In the valued field context (Section~\ref{sec:bvf}) we will deal with the  sequence $1\to \key^\times\to \RV\setminus \set 0\to \Gam\to 0$, which is pure since $\Gamma$ is torsion-free. The extra structure in $L_\mathrm{ac}^*$ is induced by the field structure on $\key$ and the order on $\Gam$. The sorts $\mathrm{A}_\phi$ are in this case $\key^\times/(\key^\times)^n$.
\end{eg}

 We may and will assume that, for each variable $x$ from an $\mathcal{A}$-sort $\mathrm{A}_s$, the formula  $\phi\coloneqq x=0$ is in $\mathcal F$, and identify $\mathrm{A}_s$ with $\mathrm{A}_\phi= \mathrm{A}_s/0\mathrm{A}_s$. In other words,   $\mathcal{A}\subseteq\mathcal A_\mathcal F$.

\begin{rem}\label{rem:split}
As pp formulas commute with cartesian products, every split short exact sequence is pure. Since purity is first-order, a short exact sequence is pure in case some elementarily equivalent structure splits.
Note that, even if a short exact sequence splits, it need not do so definably,
and that the definition of expanded pure short exact sequence does \emph{not} allow to add  splitting maps. If we add one, then matters simplify considerably. E.g., if in $L_\mathrm{ac}^*$ there is no symbol involving $\mathcal A$ and $\mathcal C$ jointly,   a splitting map makes the short exact sequence interdefinable with the disjoint union of $\mathcal A$ and $\mathcal C$, where $\invtilde$ decomposes as a product (Example~\ref{eg:disjunion}).
  \end{rem}

\begin{fact}[{\cite[Remark~4.21]{acgz20}}]\label{fact:sesqe}
Let $\phi(x^\mathrm{a}, x^\mathrm{b}, x^\mathrm{c})$ be an $L_\mathrm{abcq}^*$-formula with $x^\mathrm{a},x^\mathrm{b}$, $x^\mathrm{c}$  tuple of variables from the $\mathcal A_\mathcal F$-sorts, $\mathcal{B}$-sorts and  $\mathcal{C}$-sorts respectively. There are an $L_\mathrm{acq}^*$-formula $\psi$ and special terms $\sigma_i$ such that, in the  $L_\mathrm{abcq}^*$-theory of all expanded pure short exact sequences, we have $\phi(x^\mathrm{a}, x^\mathrm{b}, x^\mathrm{c})\coimplica \psi(x^\mathrm{a}, \sigma_1(x^\mathrm{b}),\ldots,\sigma_m(x^\mathrm{b}), x^\mathrm{c})$.
\end{fact}
\begin{co}\label{co:fullembses}
  The $L_\mathrm{acq}^*$-reduct is fully embedded. In particular,  $\mathcal{A}$ and $\mathcal{C}$ are orthogonal if and only if they are such in the $L_\mathrm{acq}^*$-reduct.
\end{co}

 We show that expanded pure short exact sequences are controlled, domination-wise, by their $L_\mathrm{acq}^*$-part, provided we pass to \textasteriskcentered-types.  This is a necessity since, in general, there are finite tuples from $\mathcal{B}$ that cannot be domination-equivalent to any finitary tuple from the $L_\mathrm{acq}^*$-reduct; see Remark~\ref{rem:wesast}. 

\begin{pr}\label{pr:esesred}
In an expanded pure short exact sequence of $L$-abelian structures, let $\mathcal F$ be a fundamental family for $\mathcal{B}$, and let $\kappa \ge \abs L$ be a small cardinal. There is a family of $\kappa$-tuples of definable functions $\set{\tau^p\mid p\in S_{\kappa}(\monster)}$ such that:
    \begin{enumerate}    \item   Each function in $\tau^p$  is defined at realisations of $p$.
    \item Each $\tau^p$ is partitioned as $(\rho^p, \nu^p)$, where   each function in $\rho^p$ is either the identity  on some $\mathrm{A}_\phi$, or has domain a cartesian product of  $\mathcal{B}$-sorts and codomain one of the $\mathrm{A}_\phi$, and each function in $\nu^p$ is either the identity on a $\mathcal{C}$-sort, or one of the $\nu_s$.
    \item For each $p\in S_{\kappa}(\monster)$ we have $p\domeq \tau^p_*p$.
    \item\label{point:okwithotimes} For each $p_0,p_1\in \invtypes_{\kappa}(\monster)$ we have $p_0\otimes p_1\domeq \tau^{p_0}_*p_0\otimes \tau^{p_1}_*p_1$.
    \end{enumerate}
\end{pr}
\begin{proof}
  Let $abc\models p(x^\mathrm{a},x^\mathrm{b},x^\mathrm{c})$, in the notation of Fact~\ref{fact:sesqe}.    Define the tuples $\nu^p$  and $\rho^p$ as follows.
 For each coordinate in $x^\mathrm{c}$ of sort $\mathrm{C}_s$, put in $\nu^p$ the corresponding identity map on $\mathrm{C}_s$.
 For each coordinate in $x^\mathrm{b}$ of sort $\mathrm{B}_s$, put in $\nu^p$ the corresponding map $\nu_s\from \mathrm{B}_s\to \mathrm{C}_s$.
 For each coordinate in $x^\mathrm{a}$ of sort $\mathrm{A}_\phi$, put in $\rho^p$ the corresponding identity map on $\mathrm{A}_\phi$.
 For each finite tuple of $L_\mathrm{b}$-terms $t(x^\mathrm{b}, w)$ and $\phi\in \mathcal F$, if there is $d\in \monster$ such that $p\proves t(x^\mathrm{b}, 0)-d\in \nu\inverse(\phi(\mathcal{C}))$, choose such a $d$, call it $d_{p,\phi,t,x^\mathrm{b}}$, and put in $\rho^p$ the map $\rho_\phi(t(x^\mathrm{b}, 0)-d_{p,\phi,t,x^\mathrm{b}})$.

Let $\tau^p$ be the concatenation of $\rho^p$ and $\nu^p$, let $q(y)\coloneqq\tau_*^p p(x)$, let $D_p$ be the set of all $d_{p,\phi, t,x^\mathrm{b}}$ as above, and let $r(x,y)\in S_{pq}(D_p)$  contain $y=\tau(x)$. Clearly $p\cup r\proves q$. By Fact~\ref{fact:sesqe}, to show $q\cup r\proves p$ it suffices to prove that $q\cup r$ recovers the formulas $\phi(x^\mathrm{a},d^\mathrm{a}, \sigma_1(x^\mathrm{b},d^\mathrm{b}),\ldots
  \sigma_m(x^\mathrm{b}, d^\mathrm{b}), x^\mathrm{c}, d^\mathrm{c})$ implied by $p$, where the $\sigma_i$ are special terms, $\phi$ is an $\mathcal L_{\mathrm{acq}}^*$-formula, and the $d^\bullet$ are tuples of parameters from the appropriate sorts of $\monster$.  Let us say that $q\cup r$ \emph{has access to} the term (with parameters) $\sigma(x^\mathrm{b},d)$ iff for some $\monster$-definable function $f$ we have $q(y)\cup r(x,y)\proves f(y)=\sigma(x^\mathrm{b},d)$. We show that $q\cup r$ has access to all special terms with parameters, hence $q\cup r\proves p$.

By construction, $q\cup r$ has access to each $\nu_s(x^\mathrm{b}_i)$. Because $\nu$ is a homomorphism of $L$-structures, $q\cup r$ also has access to each $\nu(t_0(x^\mathrm{b},d))$, for $t_0$ an $L_\mathrm{b}$-term. In particular, $q\cup r$ decides whether a given tuple $t(x^\mathrm{b},d)$ of $L_\mathrm{b}$-terms with parameters is in $\nu\inverse(\phi(\mathcal{C}))$ or not. If not, then $q\cup r$ entails  $\rho_\phi(t(x^\mathrm{b},d))=0$.

If instead $q\cup r \proves t(x^\mathrm{b},d)\in \nu\inverse(\phi(\mathcal{C}))$, by Remark~\ref{rem:termhom} $t(x^\mathrm{b},d)=t(x^\mathrm{b},0)+t(0,d)$, and by construction and the fact that $p$ is consistent with $q\cup r$ we have that  $p$ entails $t(x^\mathrm{b},0)-d_{p,\phi,t,x^\mathrm{b}}\in \nu\inverse(\phi(\mathcal{C}))$.  As this formula is over $D_p$, it is in $r$. Hence
\[
  q\cup r\proves t(0,d)+d_{p,\phi,t,x^\mathrm{b}}= t(x^\mathrm{b},0)+t(0,d)-(t(x^\mathrm{b},0)-d_{p,\phi,t,x^\mathrm{b}})\in \nu\inverse(\phi(\mathcal{C}))
\]
But $t(0,d)+d_{p,\phi,t,x^\mathrm{b}}\in \monster$, and $\rho_\phi\restr\nu\inverse(\phi(\mathcal{C}))$ is a homomorphism of $L$-structures. Because of this, and because $q\cup r$ has access to $\rho_\phi(t(x^\mathrm{b},0)-d_{p,\phi,t,x^\mathrm{b}})$ by construction, it also has access to $\rho_\phi(t(x^\mathrm{b},0)-d_{p,\phi,t,x^\mathrm{b}})+\rho_\phi(t(0,d)+d_{p,\phi,t,x^\mathrm{b}})=\rho_\phi(t(x^\mathrm{b},d))$. 

We are left to prove~\ref{point:okwithotimes}. By definition of $\otimes$, if ${p_0}(x)\otimes {p_1}(y)\proves t(x^\mathrm{b},y^\mathrm{b},d)\in \nu\inverse(\phi(\mathcal{C}))$, then there is $\tilde b\in \monster$ with ${p_0}(x)\proves t(x^\mathrm{b},\tilde b, d)\in \nu\inverse(\phi(\mathcal{C}))$. Hence, by arguing as above, ${p_0}\proves t(x^\mathrm{b},0,0)-d_{{p_0},\phi,t,x^\mathrm{b}}\in \nu\inverse(\phi(\mathcal{C}))$. So ${p_0}(x)\otimes {p_1}(y)$ entails
\[
 \nu\inverse(\phi(\mathcal{C}))\owns t(x^\mathrm{b},y^\mathrm{b},d)-t(x^\mathrm{b},0,0)+d_{{p_0},\phi,t,x^\mathrm{b}}= t(0,y^\mathrm{b},0)+t(0,0,d)+d_{{p_0},\phi,t,x^\mathrm{b}}
\]
and because $t(0,0,d)+d_{{p_0},\phi,t}\in \monster$, by construction we have ${p_1}(y)\proves t(0,y^\mathrm{b},0)-d_{{p_1},\phi,t,y^\mathrm{b}}\in\nu\inverse(\phi(\mathcal{C}))$. 
Similar arguments show that, in order to have access to $\rho_\phi(t(x^\mathrm{b},y^\mathrm{b},d))$, it is enough to have access to $\rho_\phi(t(x^\mathrm{b},0,0)-d_{{p_0},\phi,t,x^\mathrm{b}})$ together with $\rho_\phi(t(0,y^\mathrm{b},0)-d_{{p_1},\phi,t,y^\mathrm{b}})$, and the conclusion follows.
\end{proof}

\begin{co}[{Theorem~\ref{athm:ses}}]\label{co:decLtps}
Suppose that $\monster$ is an expanded pure short exact sequence of $L$-abelian structures and $\kappa\ge \abs L$ is a small cardinal.
  \begin{enumerate}
  \item There is an isomorphism of posets $\invtildestarof{\kappa}{\monster}\cong \invtildestarof{\kappa}{\monster\restr L_\mathrm{acq}^*}$.
  \item\label{point:decLtpscomptrans} If $\otimes$ respects $\doms$ in $\monster\restr L_\mathrm{acq}^*$, then the same is true in $\monster$, and the above is also an isomorphism of monoids.
  \item  If  $\mathcal{A}$ and $\mathcal{C}$ are orthogonal, then there is an isomorphism of posets $\invtildestarof{\kappa}\monster\cong \invtildestarof{\kappa}{\operatorname{\mathcal A}_\mathcal F(\monster)}\times \invtildestarof{\kappa}{\operatorname{\mathcal C}(\monster)}$.
Moreover, if $\otimes$ respects $\doms$ in  both $\operatorname{\mathcal A}_\mathcal F(\monster)$ and $\operatorname{\mathcal C}(\monster)$, then the same is true in $\monster$, and the above is also an  isomorphism of  monoids.
  \end{enumerate}
\end{co}
\begin{proof}
 By  Fact~\ref{fact:fullembemb} we have an embedding of posets $\invtildestarof{\kappa}{\monster\restr L_\mathrm{acq}^*}\into \invtildestarof{\kappa}{\monster}$. This embedding is surjective by Proposition~\ref{pr:esesred}, its inverse being induced by the maps $\tau$, hence an isomorphism.
  For~\ref{point:decLtpscomptrans}, by Proposition~\ref{pr:esesred} we may apply  Proposition~\ref{pr:comptransfer} to the family of sorts $\mathcal A_\mathcal F \mathcal C$.
We conclude by combining~\ref{point:decLtpscomptrans} with Corollary~\ref{co:orthsortsiso}.
\end{proof}
\begin{rem}\label{rem:variants}
Variants of  Fact~\ref{fact:sesqe} for settings such as abelian groups augmented by an absorbing element are presented in~\cite[Section~4]{acgz20}. These yield variants of  Proposition~\ref{pr:esesred} and its consequences, with no significant difference in the proofs.
\end{rem}
 Specialised to abelian groups, the results above enjoy a form of local finiteness.
\begin{notation}\label{notation:rhon}
For the rest of the section, $L$ is just the language of abelian groups, and $\mathcal F$ the family of formulas $\set{\exists y\; x=n\cdot y\mid n\in \omega}$. We will write $\rho_n\from \mathrm{B}\to \mathrm{A}/n\mathrm{A}$ in place of $\rho_\phi\from \mathrm{B}\to \mathrm{A}_\phi$, and identify $\mathrm{A}$ with $\mathrm{A}/0\mathrm{A}$ for notational convenience.
\end{notation}
\begin{defin}
  A \textasteriskcentered-type $p(x)$ is \emph{locally finitary} iff $x$ has finitely many coordinates of each sort.
\end{defin}
\begin{pr}\label{pr:eqsttp}
Consider a pure short exact sequence of abelian groups equipped with an $L_\mathrm{abcq}^*$-structure. Let $p(x)$ be a locally finitary global type. Then, in Proposition~\ref{pr:esesred}, we may choose  $\tau^p$ in such a way that $\tau^p_* p$ is locally finitary. 
\end{pr}
\begin{proof}
Write $p(x)=p(x^\mathrm{a},x^\mathrm{b},x^\mathrm{c})$ as in the proof of Proposition~\ref{pr:esesred}, and recall that an $L$-term is just a $\mathbb Z$-linear combination.
For each $n\in \omega$,  consider the subgroup 
  \[
K_n^p\coloneqq    \set{k\in \mathbb Z^{\abs {x^\mathrm{b}}}\mid \exists d\in \operatorname{B}(\monster)\; p\proves k\cdot x^\mathrm{b}-d \in \nu\inverse(n\mathrm{C})}
  \]
  of $\mathbb Z^{\abs {x^\mathrm{b}}}$, say generated by $k^n_0,\ldots, k^n_{m(n)}$. Choose $d_{p,n,i}$ witnessing $k^n_i\in K_n^p$.   Proceed as in Proposition~\ref{pr:esesred} but, instead of putting in  $\rho^p$ each $\rho_\phi(t(x^\mathrm{b},0)-d_{p,\phi,t})$,  use a locally finite $\rho^p$ extending  $(\rho_n(k^n_i\cdot x^\mathrm{b}-d_{p,n,i}))_{n\in \omega, i\le m(n)}$. Besides this,  $\tau^p$ contains a finite tuple of identity maps and finitely many $\nu$, therefore $\tau^p_*p$ is locally finitary. 

The proof of Proposition~\ref{pr:esesred} now goes through, with a pair of modifications which we now sketch. The first one concerns proving access to each $\rho_n(t(x^\mathrm{b},d))$. 
Fix $n$ and $t(x^\mathrm{b},d)$. Without loss of generality $d$ is a singleton and $t(x^\mathrm{b},d)=\ell \cdot x^\mathrm{b}-d$.
 If $p\proves t(x^\mathrm{b},d)\in \nu\inverse(n\mathrm{C})$, by definition we have  $\ell\in K_n^p$, so we may write $\ell=\sum_{i\le m(n)} e_i k^n_i$ for suitable $e_i\in \mathbb Z$. This allows us to rewrite
\[
  t(x^\mathrm{b},d)=\ell\cdot x^\mathrm{b}-d=\left(\sum_{i\le m(n)} e_i k^n_i\right)\cdot x^\mathrm{b}-d=\sum_{i\le m(n)} e_i (k^n_i\cdot x^\mathrm{b}-d_{p,n,i})+\sum_{i\le m(n)} e_id_{p,n,i}-d
\]
Since $\ell\cdot x^\mathrm{b}-d$ and all $k^n_i\cdot x^\mathrm{b}-d_{p,n,i}$ are in $\nu\inverse(n\mathrm{C})$, so is $\sum_{i\le m(n)} e_id_{p,n,i}-d$. Since $\rho_n\restr \nu\inverse(n\mathrm{C})$ is a homomorphism and $\sum_{i\le m(n)} e_id_{p,n,i}-d\in \monster$, we have that $q\cup r$ has access to $\rho_n(t(x^\mathrm{b},d))$.

Finally, proving point~\ref{point:okwithotimes} of Proposition~\ref{pr:esesred} boils down to showing $K_n^{p\otimes q}=K_n^p\times K_n^q$, where we identify e.g.~$K_n^p$ with $K_n^p\times \set{0}$. Since by construction $K_n^p\cap K_n^q=\set{0}$, one only needs to show generation. We leave the easy proof to the reader. 
\end{proof}
\begin{rem}
In the case of abelian groups, we therefore have an analogue of Corollary~\ref{co:decLtps} where $\kappa$-types are replaced by locally finitary $\omega$-types.
\end{rem}
\begin{co}\label{co:abgrpses}
Let $\monster$ be an expanded pure short exact sequences of abelian groups where, for all $n>0$, the sort $\mathrm{A}/n\mathrm{A}$ is finite. If $\mathrm{A}$ and $\mathrm{C}$ are orthogonal, there is an isomorphism of posets $\invtilde\cong \invtildeof{\operatorname{A}(\monster)}\times \invtildeof{\operatorname{C}(\monster)}$. 
If $\otimes$ respects $\doms$ in $\operatorname{A}$ and $\operatorname{C}$, then $\otimes$ respects $\doms$, and the above is an isomorphism of monoids.
\end{co}
\begin{proof}
Use Proposition~\ref{pr:eqsttp} and observe that for each $p$ we may replace $\tau^p$ by its composition with the projection on the nonrealised coordinates of $\tau^p_*p$ and still have the same results. If  $\mathrm{A}/n\mathrm{A}$ is finite for all $n>0$ and $p$ is finitary, this yields another finitary type. The conclusion now follows as in the proof of Corollary~\ref{co:decLtps}.
\end{proof}

\begin{rem}
The $A/nA$ are in general necessary to obtain a product decomposition. E.g., let $\mathrm{A}$ be a
    regular oag divisible by all $\primo\in \mathbb P\setminus \set 2$, and with $[\mathrm{A}:2\mathrm{A}]$
    infinite, and let $\mathrm{C}$ be a nontrivial divisible oag.   
    The expanded
    short exact sequence $0\to \mathrm{A}\to \mathrm{B}\to \mathrm{C}\to 0$ induces a group ordering on $\mathrm{B}$ (Example~\ref{eg:ses1}).  
Let $p(y)$ concentrate on
    $\mathrm{B}$, at $+\infty$,  in a new coset modulo $2\mathrm{B}$. For every nonrealised $1$-type $q$ of an element of sort $\mathrm{A}$ divisible by all $n$, we have $p\wort q$.  It follows that $p$    cannot dominate any nonrealised $p'$ in a cartesian power of $\mathrm{A}$:    such a $p'$ must have a coordinate in a nonrealised cut, and hence    dominate a type $q$ as above.     Hence, if we had a product
    decomposition as in Corollary~\ref{co:abgrpses}, then $p$ would be
    domination-equivalent to a type in a cartesian power of $\mathrm{C}$. This
    is a contradiction, because $\mathrm{C}$ is orthogonal to
    $(\mathrm{A}/n\mathrm{A})_{n<\omega}$, while $p$ dominates a nonrealised type in
    $\mathrm{A}/2\mathrm{A}$.
  \end{rem}
  \begin{rem}\label{rem:wesast}
  Analogously,  $\omega$-types are a necessity: let $\mathrm{A}$ be a regular oag with each $[\mathrm{A}:n\mathrm{A}]$ infinite, $\mathrm{C}$ a nontrivial divisible oag, and take as $p\in S_\mathrm{B}(\monster)$ the type at $+\infty$ in a new coset of each $n\mathrm{A}$. For each $n>1$, there is a nonrealised $1$-type  $q_n$  of sort $\mathrm{A}/n\mathrm{A}$ such that $p\doms q_n$. One shows that the only way for a finitary type in $((\mathrm{A}/n\mathrm{A})_{n\in \omega}, \mathrm{C})$ to dominate all of the $q_n$ is to have a nonrealised coordinate in the sort $\mathrm{A}$, hence to dominate a type orthogonal to $p$.
\end{rem}

\section{Finitely many definable convex subgroups}\label{sec:owfmig}
Using the previous two sections we may describe $\invtilde$ in oags with finitely many $L_\mathrm{oag}$-definable convex subgroups.  The arguments still work if the subgroups are defined ``by fiat'' using additional predicates, so we work in this setting.

\begin{defin}\label{defin:finintsgr}
 Let $G$ be an oag with unary predicates $H_0,\ldots, H_s$,  each defining a convex subgroup, with $0=H_0\subsetneq H_1\subsetneq\ldots\subsetneq H_{s-1}\subsetneq H_s= G$, and such that $G$ has no other definable convex subgroup. Denote by $\mathbb T$ the union of the set of prime powers with $\set{0}$ and work with the following sorts.
 For $0\le i< s$, a sort $\mathrm{S}_i$ for $G/H_{i}$, carrying $L_\mathrm{oag}$ together with predicates for $H_j/H_i$, for $i<j<s$.
 For $1\le i\leq s$ and $n\in \mathbb T$, sorts $\mathrm{Q}_{i,n}$ for $H_{i}/(nH_{i}+H_{i-1})$, carrying $L_\mathrm{ab}$. We denote by $\mathcal Q_i$ the family of sorts $(\mathrm{Q}_{i,n})_{n\in \mathbb T}$.
  We include the canonical projection and inclusion maps together with, for each $n\in \mathbb T$ and $1\le i\le s-1$, the maps $\rho_{n,i}\from \mathrm{S}_{i-1}\to \mathrm{Q}_{i,n}$ as in Notation~\ref{notation:rhon}, relative to the short exact sequence $0\to \mathrm{Q}_{i,0}\to \mathrm{S}_{i-1}\to \mathrm{S}_i\to 0$.
\end{defin}
For $1\le i<s$ the short exact sequence  $0\to H_i/H_{i-1}\to G/{H_{i-1}}\to G/{H_i}\to 0$ is pure and, as pointed out in Example~\ref{eg:ses2}, interdefinable with an expanded pure short exact sequence of abelian groups.
\begin{lemma}
Every $H_{i+1}/H_i$ is regular. For each $i\ne j$, the sort $\mathrm{S}_i$ is fully embedded as an oag, the family $\mathcal Q_i$ (with $L_\mathrm{oag}$-structure on $\sort_{i,0}$, $L_\mathrm{ab}$-structure on other sorts, and projection maps) is fully embedded,  orthogonal to $\mathrm{S}_i$, and orthogonal to $\mathcal Q_j$.
\end{lemma}
\begin{proof}
Apply Fact~\ref{fact:regoag} to $H_{i+1}/H_i$ , whose only definable convex subgroups are itself and $\set 0$. The rest is by  Corollary~\ref{co:fullembses}, Remark~\ref{rem:wort}, and induction on $i$.  
\end{proof}

\begin{thm}\label{thm:NEWfdcsg}
Let $G$ be as in Definition~\ref{defin:finintsgr}, and $\kappa$  a small infinite cardinal. Then  $\otimes$ respects $\doms$, and $\invtildestarof{\kappa}{\monster^\eq}\cong \prod_{i=1}^{s} \invtildestarof{\kappa}{\operatorname{\mathcal Q}_i(\monster)}$.
\end{thm}

\begin{proof}
 By the previous lemma, Corollary~\ref{co:decLtps}, Corollary~\ref{co:isro} and induction we get that $\otimes$ respects $\doms$, and $\invtildestarof{\kappa}{\monster}\cong \prod_{i=1}^{s} \invtildestarof{\kappa}{\operatorname{\mathcal Q}_i(\monster)}$.
  
 If the $H_i$ are $L_\mathrm{oag}$-definable,\footnote{Oags with finitely many definable convex subgroups are known as the oags \emph{of finite regular rank}. Note that every $H_i$ must be fixed by every automorphism, and is therefore $\emptyset$-definable.}  a result of Vicar\'ia (\cite[Theorem~5.1]{vicEI})  yields weak elimination of imaginaries in the language with sorts $\mathrm{S}_i/n \mathrm{S}_i$ for $0\le i < s$ and $n\in \mathbb T$,\footnote{Vicar\'ia uses sorts indexed by $n\in \omega$; as in Remark~\ref{rem:ppenough}, it suffices to work with $n\in \mathbb T$.} and one may check that her proof goes through also in the case where the $H_i$ are explicitly named by predicates, i.e.~not necessarily $L_\mathrm{oag}$-definable.

After adding the sorts from Vicar\'ia's result, for $1\le i\le s$  the short exact sequences $0\to \mathrm{Q}_{i,n}\to \mathrm{S}_{i-1}/n \mathrm{S}_{i-1}\to \mathrm{S}_i/n \mathrm{S}_i\to 0$ are fully embedded, and Corollary~\ref{co:decLtps} may thus be applied to these. From this, we obtain an embedding     $\prod_{i=1}^{s} \invtildestarof{\kappa}{\operatorname{\mathcal Q}_i(\monster)}\into  \invtildestarof{\kappa}{\monster^\eq}$. 
As $\mathrm{Q}_{s,n}=\mathrm{S}_{s-1}/n \mathrm{S}_{s-1}$, by induction on $i$ one obtains surjectivity of this embedding. We leave to the reader to check this, along with the proof of transfer of compatibility of $\otimes$ and $\doms$, by showing that every \textasteriskcentered-type is dominated by its image among a suitable tuple of definable maps.
\end{proof}

\section{Benign valued fields}\label{sec:bvf}
In this section $T$ is a complete $\RVses$-expansion of a theory of henselian valued fields with elimination of $\K$-quantifiers and ``enough maximal saturated models'' 
(see below for the precise definitions). We show the existence of an isomorphism $\invtilde\cong\invtildeof{\RVsesof(\monster)}$. In particular, our results hold in any \emph{benign} valued field in the sense of~\cite{touchardburden}\footnote{\cite[Definition~1.57]{touchardburden} allows $\set k\textrm{-}\set\Gam$-expansions in the definition of benign. Since we are shortly going to allow more general expansions, the difference is immaterial for our purposes.}, i.e., in any henselian valued field which is of equicharacteristic $0$, or algebraically closed, or algebraically maximal Kaplansky of characteristic $\primo>0$.

Associate to a valued field $\K$  the pure (Example~\ref{eg:ses2}) short exact sequence  $1\to \key^\times\to \K^\times/(1+\mathfrak m)\to \Gam\to 0$.
 Add absorbing elements $0$, $0$, $\infty$, 
 and view it as a short exact sequence of abelian monoids   $1\to \key\to \K/(1+\mathfrak m)\to \Gam\cup\set{\infty}\to 0$. We may  harmlessly conflate the two settings (Remark~\ref{rem:variants}) and write $\Gamma$ for $\Gamma\cup\set\infty$.

The middle term $\K/(1+\mathfrak m)$ is called the \emph{leading term structure}  $\RV$, and comes with a natural map $\rv\from \K\to \K/(1+\mathfrak m)=\RV$ through which the valuation $v\from \K\to \Gam$ factors.
Besides the structure of a (multiplicatively written) monoid, $\RV$ is equipped with a ``partially defined sum'': a ternary relation defined as $\oplus(x_0,x_1,x_2)\overset{\mathrm{def}}{\iff} \exists y_0, y_1, y_2\in \K\; \left(y_2=y_0+y_1\land \bigwedge_{i<3}\rv(y_i)=x_i\right)$.
When there is a unique $x_2$ such that $\oplus(x_0,x_1,x_2)$, we write $x_0\oplus x_1=x_2$, and say that $x_0\oplus x_1$ \emph{is well-defined}. It turns out that $\rv(x)\oplus\rv(y)$ is well-defined if and only if $v(x+y)=\min\set{v(x), v(y)}$. If we say that $\bigoplus_{i<\ell}{x_i}$ \emph{is well-defined}, we mean that, regardless of the choice of parentheses and order of the summands,  the ``sum'' is well-defined and always yields the same result.

Let $\RVses$  be the expansion of
$1\to \key\xrightarrow{\iota} \RV\xrightarrow{v} \Gam\to 0$ by the field structure on $\key$ and the order on $\Gam$.
 This induces an expansion of $\RV$, which is precisely that given by multiplication and  $\oplus$~\cite[Lemma~5.17]{acgz20},
 is biinterpretable with $\RVses$, and can be axiomatised independently~\cite[Appendix~B]{touchardburden}. Hence, we may view $\RV$ as a standalone structure $(\RV, \cdot, \oplus)$, fully embedded in $(\K, \RV, \rv)$, and in $\RVses$.

By the (Short) Five Lemma, an extension of valued fields is immediate, i.e.~does not change $\key$ nor $\Gam$, if and only if it does not change $\RVses$.

In this section,   $L$ has sorts  $\K, \key, \RV, \Gam$, function symbols $\rv\from \K\to \RV$, $\iota\from \key\to \RV$, $v\from \RV\to \Gam$. We abuse the notation and also write $v$ for the composition $v\circ \rv$.
The sorts $\K$ and $\key$ carry disjoint copies of the language of rings, 
 $\Gam=\Gam\cup\set\infty$ carries the (additive) language of ordered groups, together with an absorbing element $\infty$ and an extra constant symbol $v(\operatorname{Char}(\K))$, and $\RV$ carries the (multiplicative) language of groups, together with an absorbing element $0$ and a ternary relation symbol $\oplus$.
 We denote by $\RVses$ the reduct to the sorts $\key,\RV,\Gam$.
 There may be other arbitrary symbols on $\RVses$, i.e., as long as they do not involve $\K$.
 An \emph{$\RVses$-expansion} of a theory $T'$ of valued fields is a complete $L$-theory $T\supseteq T'$. Until the end of the section, $T$ denotes such a theory.
We identify  $\key$ with the image of its embedding $\iota$ in $\RV$. 
\begin{rem}\label{rem:denefpas}
 Angular components factor through the map $\rv$, yielding a splitting of $\RVses$. Therefore,
the Denef--Pas language (and  each of its  $\set{\key,\Gam}$-expansions\footnote{A $\set{\key,\Gam}$-expansion is one where the new symbols only involve the sorts $\key$ and $\Gamma$, possibly simultaneously. If we want to exclude the latter possibility, we speak of $\set{\key}\textrm{-}\set{\Gam}$-expansions.}) may be seen as an $\RVses$-expansion.  In that case $\RVses$ is definably isomorphic to $\key\times \Gamma$.
\end{rem}

\begin{fact}\label{fact:rvqe}
 Fix a language $L$ as above.   The theory of all $\RVses$-expansions of benign valued fields eliminates $\K$-sorted quantifiers. \end{fact}
\begin{proof}
In equicharacteristic this follow from \cite[Th\'eor\`{e}me~2.1]{Del-These}. The residue characteristic $0$ case is explicitly done in \cite[Theorem~B]{Basarab_1991} (see also \cite[Corollary~2.2]{Kuh}), the algebraically maximal Kaplansky case in \cite[Theorem~2.6]{Kuh} (see~\cite[Corollary~A.3]{hhqe} for a modern treatement). The algebraically closed case is folklore (see, e.g., \cite[Fact~2.4]{imagscvf}).
\end{proof}

\begin{rem}
If $T$ eliminates $\K$- quantifiers, then every formula is equivalent to one of the form $\phi(x,\rv(f_0(y)),\ldots,\rv(f_m(y)))$, 
where  $\phi(x,\bla z0,m)$ is a formula in $\RVses$, $x$ and $z$ tuples of $\RVses$-variables,  $y$ a tuple of $\K$-variables, and the $f_i$ polynomials over $\mathbb Z$. In particular, $\RVses$ (with the restriction of $L$ to its sorts) is fully embedded.
\end{rem}
\begin{proof}
  By inspecting the formulas without $\K$-sort quantifiers and observing that, for example, if $y$ is of sort $\K$ then $T\proves y=0\coimplica \rv(y)=0$.
\end{proof}
\begin{defin}
  Let $K_0\subseteq K_1$ be an extension of valued fields. A basis $(a_i)_{i}$ of a $K_0$-vector subspace of $K_1$ is \emph{separating} iff for all finite tuples $d$ from $K_0^\ell$  and pairwise distinct $i_j$, we have $v\left(\sum_{j<\ell}d_j a_{i_j}\right)=    \min_{j<\ell} \bigl(v(d_j)+v(a_{i_j})\bigr)$. If every finite dimensional $K_0$-vector subspace has such a basis, the extension is called \emph{separated}.
\end{defin}

\begin{fact}\label{fact:seprv}
  A basis $(a_i)_{i}$ is separating if and only if each sum 
  \(
    \bigoplus_{j<\ell} \rv(d_j)\rv(a_{i_j})
  \)
  is well-defined. If this is the case, it equals $\rv(\sum_{j<\ell}d_j a_{i_j})$.
\end{fact}
\begin{lemma}\label{lemma:seplift}
  Let $p\in \invtypes_{\K^{\le\omega}}(\monster, M_0)$,   $M_0\preceq M\smallprec \monster\subseteq B$,   $a\models p\invext B$, and $(f_i)_{i\in I}$  a family of $M$-definable functions $\K^{\omega}\to \K$ such that $(f_i(a))_{i\in I}$ is a separating basis of the $\Kof(M)$-vector space they generate.
If $M$ is  $\abs{M_0}^+$-saturated, or $p$ is definable,  then  $(f_i(a))_{i\in I}$ is a separating basis of the $\Kof(B)$-vector space they generate.
\end{lemma}
\begin{proof}
  Towards a contradiction, suppose there are an  $L(M)$-formula    \[
    \phi(x, w)\coloneqq v\left(\sum_{i<\ell} w_i f_i(x)\right)> \min_{i<\ell} \set{v(w_i)+v(f_i(x))}
  \]
 and $d\in B^{\abs{w}}$ such that $a\models \phi(x,d)$. Let $H$ be the set of parameters  appearing in
 $\phi(x,w)$. Choose $\tilde d\in M$  with $\tilde d\equiv_{M_0H} d$ if $M$ is   $\abs{M_0}^+$-saturated, or in  $d_p\phi$ if $p$ is definable. Then $a\models \phi(x,\tilde d)$ contradicts that $(f_i(a))_{i\in I}$ is  separating over $M$.
\end{proof}
Hence, saturation of $M$ allows to lift separating bases. As maximality of $M$ guarantees their existence (see Lemma~\ref{lemma:sepbas} below), we give the following definition.

\begin{defin}
We say that  $T$ has \emph{enough saturated maximal models} iff for every $\kappa>\abs L$, for every $M_0\models T$ of size at most $\kappa$ there is $M\succ M_0$ of size at most $2^{2^\kappa}$ which is maximally complete and $\abs{M_0}^+$-saturated. 
\end{defin}
\begin{rem}
If we restrict to definable types, saturation is not necessary to lift separating bases (cf.~Lemma~\ref{lemma:seplift}), and it is enough to assume only
    ``enough maximal models'' for weak versions of the results of this
    section to go through.
  \end{rem}
\begin{pr}\label{pr:maxsat}
Let $T$ be an $\RVses$-expansion of a theory of henselian valued fields eliminating $\K$-quantifiers, where every $M\models T$ has a unique maximal immediate extension up to isomorphism over $M$.
If $M'\models T$ is maximal, $\kappa> \abs L$, and $\RVsesof(M')$ is $\kappa$-saturated, then $M'$ is $\kappa$-saturated.   
\end{pr}
The proposition above is folklore, but we include a proof for convenience.  As pointed out to us by the referee, uniqueness of the maximal immediate extension is not needed, and maximality of $M'$ may be relaxed to requiring that chains of balls of length smaller than $\kappa$ have nonempty intersection; the result then follows by using Swiss cheese decomposition. Nevertheless, the proof below has the advantage that it can be adapted to more general contexts, which we will need in Proposition~\ref{pr:diffesmm}.
\begin{proof}
  If $\kappa$ is limit  $\kappa$-saturation equals $\lambda$-saturation for all $\lambda<\kappa$, so we may assume $\kappa$ is successor, hence regular. It suffices to prove that if $M\equiv M'$ is $\kappa$-saturated, then the set $\mathcal S$ of partial  elementary maps between $M$ and $M'$ with domain of size less than $\kappa$ has the back-and-forth property. In fact, we only need the ``forth'' part (and the ``back'' part is true by $\kappa$-saturation of $M$).
 So assume $f\in \mathcal S$, with
\[
  f\from A=(\Kof(A), \RVsesof(A))\to A'=(\Kof(A'), \RVsesof(A'))
\]
and suppose that $A\subseteq B\subseteq M$, with $\abs{B}<\kappa$. In order to extend $f$ to some $g\in \mathcal S$ with domain containing $B$, consider the following two constructions.

\textbf{Construction 1.}  Enlarge $A$ to an elementary substructure. That is, there are $A_1\supseteq A$ and $f_1\from A_1\to A_1'$  extending $f$ such that  $f_1\in \mathcal S$ and $A_1\preceq M$. To do this, we find $A_1'$ with $A'\subseteq A_1'\preceq M'$ and $\abs{A_1'}<\kappa$ using L\"owenheim--Skolem, and invoke $\kappa$-saturation of $M$ to obtain the desired $A_1, f_1$.

\textbf{Construction 2.} For a given $\widehat B$ such that $A\subseteq \widehat B\subseteq M$ and $\abs{\widehat B}<\kappa$,
enlarge $\RVsesof(A)$ so that it contains $\RVsesof(\widehat B)$. That is, there are $A_1\supseteq A$ and $f_1\from A_1\to A_1'$  extending $f$ such that  $f_1\in \mathcal S$ and $\RVsesof(A_1)\supseteq \RVsesof(\widehat B)$.  To do this, it suffices to set $A_1=(\Kof(A), \RVsesof(\widehat B))$ and extend $f$ on $\RVses$ using $\kappa$-saturation of $\RVsesof(M')$; by elimination of $\K$-quantifiers, the extension is still an elementary map.

By repeated applications of the constructions above, we find an elementary chain $(M_n)_{n\in \omega}$ of elementary submodels of $M$, with $A\subseteq M_0$, and $f_n\in\mathcal S$ with domain $M_n$ such that $f_0\supseteq f$, $f_{n+1}\supseteq f_n$, and that if $B_n$ is the structure generated by $M_nB$ then $\RVsesof(B_n)\subseteq \RVsesof(M_{n+1})$.  Let $M_\omega\coloneqq\bigcup_{n\in \omega}M_n$ and let $f_\omega\coloneqq\bigcup_{n\in\omega} f_n$. Since $\kappa$ is regular and uncountable we have  $f\in \mathcal S$, and by construction the structure $B_\omega$ generated by $M_\omega B$ is $\K$-generated and an immediate extension of $M_\omega$. Since $M'$ is maximal and the maximal immediate extension of $M_\omega$ is uniquely determined up to $M_\omega$-isomorphism, we may extend $f_\omega$ to a map $g\in \mathcal S$ with domain $B_\omega\supseteq B$.
\end{proof}
\begin{rem}
Above (and in $\set{\key}\textrm{-}\set{\Gam}$-expansions of the Denef--Pas language), if $\key$ and $\Gam$ are orthogonal it suffices to assume that $\keyof(M')$ and $\Gamof(M')$ are $\kappa$-saturated.
\end{rem}
\begin{co}\label{co:maxsat}
  Suppose that $T$ satisfies the assumptions of Proposition~\ref{pr:maxsat}, and furthermore that every maximal immediate extension of every $M\models T$ is an elementary extension.
Then $T$  has enough saturated maximal models.
\end{co}
\begin{proof}
  Given  $\kappa>\abs L$ and $M_0\models T$ of size  $\abs{M_0}\le \kappa$, find $M_1\succ M_0$ which is $\abs{M_0}^+$-saturated of size $\abs{M_1}\le 2^{{\abs{M_0}}}$. Let $M$ be a maximal immediate extension of $M_1$. Then $\RVsesof(M)=\RVsesof(M_1)$, and the latter is $\abs{M_0}^+$-saturated because $M_1$ is. By assumption, $M\succ M_1$, and by Proposition~\ref{pr:maxsat} $M$ is $\abs{M_0}^+$-saturated. To conclude, observe that, since by Krull's inequality~\cite[Proposition~3.6]{vddvf} we have $\abs{\K}\le{\key}^{\Gam}$, we obtain $\abs{M}\le \abs{\keyof(M)}^{\abs{\Gamof(M)}}=\abs{\keyof(M_1)}^{\abs{\Gamof(M_1)}}\le (2^{\abs{M_0}})^{2^{\abs{M_0}}}=2^{2^{\abs{M_0}}}$.
\end{proof}
\begin{co}
Every $\RVses$-expansion of a benign $T$ has enough saturated maximal models.
\end{co}
\begin{proof}
 Since the assumptions of Fact~\ref{fact:rvqe} are preserved by taking maximal immediate extensions (which are unique by~\cite[Theorem~5]{kaplansky})  elementarity follows from elimination of $\K$-quantifiers. We conclude by Corollary~\ref{co:maxsat}.
\end{proof}

\begin{lemma}\label{lemma:sepbas}
  Let $p,q\in \invtypes_{\K^{<\omega}}(\monster, M_0)$, let  $(a,b)\models p\otimes q$ and $M_0\prec M\smallprec \monster$. 
  \begin{enumerate}
  \item\label{point:sep1}  If $M$ is maximally complete,  then there are polynomials $(f_i)_{i<\omega}$   in $\Kof(M)[x]$ such that $(f_i(a))_{i<\omega}$ is a separating basis of $\Kof(M)[a]$ as a $\Kof(M)$-vector space.
  \item\label{point:sep2}  If $M$ is  $\abs{M_0}^+$-saturated then, for each $(f_i)_{i<\omega}$  as above, $\set{f_i(a)\mid i<\omega}$ is a separating basis of $\Kof(\monster)[a]$.
  \item\label{point:sepprod}  If  $(f^p_i(a))_{i<\omega}$, $(f^q_j(b))_{j<\omega}$ are  separating bases of $\Kof(\monster)[a]$ and $\Kof(\monster)[b]$, then $(f^p_i(a)\cdot f^q_j(b))_{i,j<\omega}$ is a separating basis of $\Kof(\monster)[ab]$.
  \end{enumerate}
\end{lemma}

\begin{proof}
 Part \ref{point:sep1} is by~\cite[Lemma~3]{baur} (see also~\cite[Lemma~12.2]{hhm})  and does not require saturation, and part \ref{point:sep2} is by Lemma~\ref{lemma:seplift} applied to $(f_i)_{i<\omega}$. So we only need to prove \ref{point:sepprod}. By the definition of $\otimes$, the tuple $(f^p_i(a)\cdot f^q_j(b))_{i,j<\omega}$ is linearly independent, and clearly it generates $\Kof(\monster)[ab]$ as a $\Kof(\monster)$-vector space. Let us check that this basis is separating. Let $B$ be the structure generated by $\monster b$. By Lemma~\ref{lemma:seplift}, $(f^p_i(a))_{i<\omega}$ is a separating basis of the $\K(B)$-vector space $\K(B)[a]$, so we have
 \begin{multline*}
   v\Bigl(\sum_{i,j} d_{ij}f^p_i(a)f^q_j(b)\Bigr)
   =v\Bigl(\sum_{i} \Bigl(\sum_{j}d_{ij}f^q_j(b)\Bigr)f^p_i(a)\Bigr)\\
   =\min_i \Bigl(v\Bigl(\sum_{j}d_{ij}f^q_j(b)\Bigr)+v(f^p_i(a))\Bigr)
   =\min_i \Bigl(\min_j\bigl(v(d_{ij})+v(f^q_j(b))\bigr)+v(f^p_i(a))\Bigr)\\
   =\min_{i,j}\Bigl(v(d_{ij})+v(f^q_j(b))+v(f^p_i(a))\Bigr)=  \min_{i,j}\Bigl(v(d_{ij})+v(f^q_j(b)\cdot f^p_i(a))\Bigr)\qedhere
 \end{multline*}
\end{proof}

\begin{pr}\label{pr:rvot}
Suppose that $T$ eliminates $\K$-quantifiers and has enough saturated maximal models.
For every $p\in \invtypes(\monster)$  there is $q\in \invtypes_{\RVses^\omega}(\monster)$ such that $p\domeq q$. More precisely,  let $p(x,z)\in\invtypes(\monster, M_0)$, where $x$ is a tuple of $\K$-variables and $z$ a tuple of $\RVses$-variables. Let $(a,c)\models p(x,z)$, let $M\succ M_0$ be $\abs{M_0}^+$-saturated and maximally complete, and let $(f_i)_{i<\omega}$ be given by  Lemma~\ref{lemma:sepbas} applied to $a$ and $M$. Then $p$ is domination-equivalent to the \textasteriskcentered-type  $q(y,t)\coloneqq \tp(\operatorname{rv}(f_i(a))_{i<\omega}, c/\monster)$, witnessed by $r(x,z,y,t)\coloneqq \tp(a,c, \operatorname{rv}(f_i(a))_{i<\omega},c/M)$.
\end{pr}
\begin{proof}
  That $p\cup r\proves q$ is trivial. By elimination of $\K$-quantifiers (Fact~\ref{fact:rvqe}), to prove $q\cup r\proves p$ it is enough to show that $q\cup r$ has access to every $\rv(f(x))$,  i.e., that for every $f\in \Kof(\monster)[x]$, there is a $\monster$-definable function $g$ such that $q\cup r\proves \operatorname{rv}(f(x))=g(y)$. 
  Write $f(x)=\sum_{i<\ell} d_i f_i(x)$. By Fact~\ref{fact:seprv}, we have $\operatorname{rv}(f(a))=\bigoplus_{i<\ell} \operatorname{rv}(d_i)\operatorname{rv}(f_i(a))$, and we only need to ensure that this information is in $q\cup r$.
But by Fact~\ref{fact:seprv} whether the $(f_i(a))_{i<\omega}$ form a separating basis or not only depends on the type of their images in $\RV$, which is part of $q$ by definition.
\end{proof}
The work done so far is enough to obtain an infinitary version of Theorem~\ref{athm:rv}. After stating such a version, we will proceed to finitise it.
\begin{rem}\label{rem:spebasstartypes}
 Separating bases of vector spaces of uncountable dimension need not exist. Nevertheless, a \textasteriskcentered-type version of Lemma~\ref{lemma:sepbas} still holds, with the $f_i(a)$ now enumerating separating bases of all finite dimensional subspaces of $\K(M)[a]$.
\end{rem}
\begin{co}\label{co:rvredinf}
If $\kappa$ is a small infinite cardinal, there is an isomorphism of posets $\invtildestarof\kappa\monster\cong\invtildestarof\kappa{\RVsesof(\monster)}$.   If $\otimes$ respects $\doms$ on \textasteriskcentered-types in $\RVsesof(\monster)$, then the same holds in $\monster$, and the above is also an isomorphism of monoids.
\end{co}
\begin{proof}
 By the \textasteriskcentered-type versions of Lemma~\ref{lemma:sepbas} and  Propositions~\ref{pr:rvot} and~\ref{pr:comptransfer}. 
\end{proof}

\begin{lemma}\label{lemma:finitise}
Let $M_0\smallprec M\smallprec \monster$,  let $e\models q\in \invtypes_{\RV^\omega}(\monster, M_0)$. Let $I\subseteq \omega$ be such that  $(v(e_i))_{i\in I}$  generates $\mathbb Q\Span{\Gamof(\monster) v(e)}$ over $\mathbb Q\Gamof(\monster)$ as $\mathbb Q$-vector spaces. Let $G\subseteq\RV$ be the multiplicative group generated by  $\RVof(\monster)e$.  Let $(g_j)_{j\in J}\subseteq \key\cap G$ be  such that  $\key\cap G\subseteq \acl(\monster (g_j)_{j\in J})$ and $J$ is countable. Let $b\coloneqq (e_i,g_j\mid i\in I, j\in J)$. Then there is $M\prec N\smallprec \monster$ such that $e$ and $b$ are interalgebraic over $N$.
\end{lemma}

\begin{proof}
By assumption, for $\ell\in \omega\setminus I$ there are $n_\ell>0$, $d_\ell\in \monster$,  a finite $I_0\subseteq I$ and, for $i\in I_0$, integers  $n_{\ell,i}\in \mathbb Z$, with $n_\ell v(e_\ell)=    v(d_\ell)+\sum_{i\in I_0} n_{\ell,i}v(e_{i})$.
  By $M_0$-invariance, we may  assume  $d_\ell\in M$. Let $h_\ell(x)$ be the $M$-definable function $h_\ell(y)\coloneqq (y_\ell^{n_\ell})/(d_\ell\prod_{i\in I_0} y_i^{n_{\ell,i}})$.
By construction, $v(h_\ell(e))=0$, hence $h_\ell(e)\in G\cap \key^\times$, so by assumption 
   $h_\ell(e)\in \acl(\monster(g_j)_{j\in J})$. Let $N\succ M$ be small such that $\set{h_\ell(e)\mid \ell\in \omega\setminus I}\subseteq \acl(N(g_j)_{j\in J})$ and $\set{g_j\mid j\in J}$ is contained in the group generated by $\RVof(N)e$.
  By definition of $h_\ell$, for each $\ell\in \omega\setminus I$, we therefore have $e_\ell^{n_\ell}\in \acl(Nb)$. A $\Gam$ is ordered and  the kernel of $v\from \RV\to \Gam$ is the multiplicative group of a field, $\RV$ has finite $n$-torsion for each $n$, hence $e_\ell$ is algebraic over $e_\ell^{n_\ell}$, hence $e\in \acl(N b)$.
\end{proof}

\begin{thm}[{Theorem~\ref{athm:rv}}]\label{thm:rvred}  For $T$ an $\RVses$-expansion of a theory of valued fields with enough saturated maximal models eliminating $\K$-quantifiers, (e.g.\ a benign one), there is an isomorphism of posets $\invtilde\cong\invtildeof{\RVsesof(\monster)}$. If $\otimes$ respects $\doms$ in $\RVsesof(\monster)$, then $\otimes$ respects $\doms$ in $\monster$, and the above is an isomorphism of monoids.
\end{thm}
\begin{proof}
Fix $p(x,z)\in\invtypes(\monster)$ and $ac\models p$, where $x$ is a tuple of  $\K$-variables  and $z$ a tuple of $\RVses$-variables.  Let $(f_i)_{i<\omega}$ be given by Lemma~\ref{lemma:sepbas}.
As usual, denote by $\monster(a)$ the field generated by $a$ over $\monster$.  As $\trdeg(\monster(a)/\monster)$ is finite, by the Abhyankar inequality so is $\dim_{\mathbb Q}(\mathbb Q\Gamof(\monster(a))/\mathbb Q\Gamof(\monster))$. Let $m$ be such that  $v(f_i(a))_{i<m}$ generates $\mathbb Q\Gamof(\monster(a))$ over $\mathbb Q\Gamof(\monster)$. Again by the Abhyankar inequality,  $\trdeg(\keyof(\monster(a))/\keyof(\monster))$ is finite.  By the choice of the $f_j$ and Fact~\ref{fact:seprv}, we may choose a transcendence basis $(g_j\mid j<n)$ of $\keyof(\monster(a))$ over $\keyof(\monster)$, which is contained in the group generated by $\RVof(\monster)(\rv(f_i(a)))_{i<\omega}$. Write each $g_j$ as $h_j(a)$, for suitable definable functions $h_j$. We may now apply  Lemma~\ref{lemma:finitise} to $e=(\rv(f_i(a)))_{i<\omega}$, the $g_j$ defined above, and $I=\set{i\in \omega \mid i<m}$.  Together with  Proposition~\ref{pr:rvot}, we obtain
  \begin{equation}
    p\domeq p'\coloneqq \tp(\rv(f_i(a))_{i<m}, (h_j(a))_{j<n},c/\monster)\label{eq:nicedomeqrv}
\end{equation}
Therefore, every (finitary) type is equivalent to one in $\RVses$. By full embeddedness of $\RVses$, and Fact~\ref{fact:fullembemb}, we obtain the required isomorphism of posets.

By Proposition~\ref{pr:comptransfer} it is enough to show that if $p',q'$ are obtained from $p,q$  as in~\eqref{eq:nicedomeqrv} above, then $p\otimes q\domeq p'\otimes q'$.
Denote by $\rho^p(x,z)\coloneqq (\rv(f^p_i(x))_{i<m_p}, (h^p_j(x))_{j<n_p},\operatorname{id}^p(z))$ the tuple of definable functions from~\eqref{eq:nicedomeqrv}, and similarly for $q$ and $p\otimes q$.
By point~\ref{point:sepprod} of Lemma~\ref{lemma:sepbas} we may take as $(f^{p\otimes q}_i)_{i<\omega}$ (a reindexing on $\omega$ of) the concatenation of $(f^p_i)_{i<\omega}$ with $(f^{q}_i)_{i<\omega}$. By the properties of $\otimes$, the concatenation of $(f^p_i(a))_{i<m_p}$ and $(f^{q}_i(b))_{i<m_q}$ is a basis of the vector space $\mathbb Q\Span{\Gamof(\monster) (v(f^p_i(a)))_{i<\omega}(v(f^q_i(b)))_{i<\omega}}$ over $\mathbb Q\Gamof(\monster)$, and so as $(f^{p\otimes q}_i)_{i<m_{p\otimes q}}$ we may take the concatenation of $(f^p_i)_{i<m_p}$ with $(f^{q}_i)_{i<m_{q}}$. Similarly, as $(h^{p\otimes q}_j)_{j<n_{p\otimes q}}$ 
we may take the concatenation of the respective tuples for $p$ and $q$, and ultimately we obtain that as $\rho^{p\otimes q}$ we may take the concatenation of $\rho^{p}$ with $\rho^{q}$. By~\eqref{eq:nicedomeqrv}, we have $p\otimes q\domeq p'\otimes q'$ and we are done.
\end{proof}

For $\set{\key, \Gam}$-expansions, we are in the setting of Section~\ref{sec:ses}, so we may combine the above with e.g.\ Theorem~\ref{athm:ses}  or Corollary~\ref{co:abgrpses}. We spell out two nice cases; the special subcases of $\mathsf{ACVF}$ and $\mathsf{RCVF}$ were previously known (see the introduction).

\begin{co}[{Theorem~\ref{athm:bendec}}]\label{co:benign}
Let $T$ be a complete $\set{\key}\textrm{-}\set{\Gam}$-expansion of a benign theory of valued fields where, for all $n>1$,  the group $\key^\times/(\key^\times)^n$ is finite.  There is an isomorphism of posets $\invtilde\cong\invtildeof{\keyof(\monster)}\times \invtildeof{\Gamof(\monster)}$.  If $\otimes$ respects $\doms$ in $\key$ and $\Gam$, then $\otimes$ respects $\doms$, and the above is an isomorphism of monoids.
\end{co}
\begin{proof}
Apply Theorem~\ref{thm:rvred}. By Fact~\ref{fact:rvqe}, if the extra structure on $\RVses$ involves only $\key$ and $\Gam$, and never both at the same time, then the sorts $\key$ and $\Gam$ are orthogonal. As $\RVses$ is an expanded pure short exact sequence,  we conclude by  Corollary~\ref{co:abgrpses}.
\end{proof}

\begin{co}\label{co:benignstar}
Let $T$ be a complete $\set{\key}\textrm{-}\set{\Gam}$-expansion of a benign theory of valued fields,  and let $\mathcal A_\mathrm{k}$ denote the family of sorts $(\key^\times/(\key^\times)^n)_{n\in \omega}$.  For  $\kappa\ge \abs L$, there is an isomorphism of posets $\invtildestarof\kappa\monster\cong\invtildestarof\kappa{\operatorname{\mathcal A_\mathrm{k}}(\monster)}\times \invtildestarof\kappa{\Gamof(\monster)}$. 
  If $\otimes$ respects $\doms$ in $\mathcal A_\mathrm{k}$ and $\Gam$, then $\otimes$ respects $\doms$, and the above is an isomorphism of monoids.
\end{co}
\begin{proof}
 As in Corollary~\ref{co:benign}, but using Corollary~\ref{co:decLtps} instead of Corollary~\ref{co:abgrpses}.
\end{proof}
 In special cases, results such as the previous corollaries may also be obtained by using domination by a family of sorts in the sense of~\cite[Definition~1.7]{ehm} (see~\cite[Section~6]{dominomin}). This kind of domination was proven in the algebraically closed case in~\cite{hhm},  in the real closed case in~\cite{ehm}, and in the equicharacteristic zero case in~\cite{vicRFD}.

In algebraically or real closed valued fields, the decomposition $\invtilde\cong\invtildeof{\keyof(\monster)}\times\invtildeof{\Gamof(\monster)}$ remains valid after passing to $T^\eq$, as can be shown using resolutions~\cite{hhm,ehm,dominomin}. A natural question is whether Theorem~\ref{thm:rvred} generalises to $T^\eq$, or at least to $T^\mathcal G$, the expansion of $T$ by the geometric sorts of~\cite{hhm_EI}.
\begin{question}
Let $T$ be an $\RVses$-expansion of a theory of valued fields with enough saturated maximal models eliminating $\K$-quantifiers. Are there conditions guaranteeing that the isomorphism $\invtilde\cong \invtildeof{\RVsesof(\monster)}$ holds in $T^{\mathcal G}$, or even in $T^\eq$? Does compatibility of $\doms$ with $\otimes$ transfer? 
\end{question}

\section{Mixed characteristic henselian valued fields}\label{sec:mixedchar}
Let $\K$ be henselian of characteristic $(0,\primo)$, for $\primo\in \mathbb P$. 
For $n\in \omega$, let $\mathfrak m_n\coloneqq\set{x\in \K\mid v(x)>v(\primo^n)}$. Let $\RV_n$  be the multiplicative monoid $\RV_n\coloneqq \K/(1+\mathfrak m_n)$, and $\RV_n^\times\coloneqq \RV_n\setminus \set 0$. For each $n$, denote by $\rv_n\from \K\to\RV_n$ the quotient map.  For $m > n$, we have natural maps $\rv_{m, n}\from \RV_m\to \RV_n$, and the valuation $v\from \K\to \Gam$ induces maps $\RV_n\to \Gam$,  still denoted by $v$. The kernel $\key_n$  of $v$ fits in a short exact sequence $1\to \key_n\to \RV_n\xrightarrow{v} \Gam\to 0$. 
We have relations $\oplus_n$, defined analogously to $\oplus$, and again well-defined precisely when $v(x+y)=\min\set{v(x), v(y)}$.  For $n=0$ we recover the notions from the previous section.  The following generalises Fact~\ref{fact:seprv}.

\begin{fact}\label{fact:seprvmix}
    A basis $(a_i)_{i}$ is separating if and only if, for each $n\in \omega$, each sum 
  \(
 \rv_n(d_0)\rv_n(a_{i_0})\oplus_n\ldots\oplus_n \rv_n(d_\ell)\rv_n(a_{i_\ell})
  \)
  is well-defined, if and only if this happens for $n=0$. If this is the case, then the sum equals $\rv_n\left(\sum_{j\le \ell}d_j a_{i_j}\right)$.
\end{fact}

In this section,   $L$ is  a language as follows.
 We have sorts $\K,\Gam$ and, for each $n\in\omega$, sorts $\key_n, \RV_n$.
 There are function symbols $\rv_n\from \K\to \RV_n$, $\iota\from \key_n\to \RV_n$, $v\from \RV_n\to \Gam$. The sort $\K$ carries a copy of the language of rings, while $\Gam=\Gam\cup\set\infty$ carries the (additive) language of ordered groups, together with an absorbing element $\infty$ and an extra constant symbol $v(\primo)$.
Each $\RV_n$ and $\key_n$ carries the (multiplicative) language of groups, together with an absorbing element $0$ and a ternary relation symbol $\oplus_n$.
 We denote by $\RVses_*$ the reduct to the sorts $\key_n,\RV_n,\Gam$.
 There may be other arbitrary symbols on $\RVses_*$, i.e., as long as they do not involve $\K$.

 An \emph{$\RVses_*$-expansion} of a theory $T'$ of henselian valued fields of characteristic  $(0,\primo)$ is a complete $L$-theory $T\supseteq T'$, with the sorts and symbols above  interpreted in the natural way. Until the end of the section, $T$ denotes such a theory.
 We will freely confuse the sort $\key_n$ with the image of its embedding in $\RV_n$. By~\cite[Theorem~B]{Basarab_1991} (see also \cite[Proposition~4.3]{flenner}) $T$ eliminates $\K$-quantifiers, so  $\RVses_*$ is fully embedded.

\begin{pr}\label{pr:rvotmix}
Suppose $T$ eliminates $\K$-quantifiers and has enough saturated maximal models.  For every $p\in \invtypes(\monster)$  there is $q\in \invtypes_{\RVses_*^\omega}(\monster)$ such that $p\domeq q$. More precisely,  let $p(x,z)\in\invtypes(\monster, M_0)$, where $x$ is a tuple of $\K$-variables and $z$ a tuple of $\RVses_*$-variables. Let $(a,c)\models p(x,z)$, let $M\succ M_0$ be $\abs{M_0}^+$-saturated and maximally complete, and let $(f_i)_{i<\omega}$ be given by the \textasteriskcentered-type version of Lemma~\ref{lemma:sepbas} applied to $a$ and $M$ (cf.~Remark~\ref{rem:spebasstartypes}). Then $p\domeq q(y,t)\coloneqq \tp(\operatorname{rv}_n(f_i(a))_{i,n<\omega}, c/\monster)$, witnessed by $r(x,z,y,t)\coloneqq \tp(a,c, \operatorname{rv}_n(f_{i}(a))_{i,n<\omega},c/M)$.
If $\kappa\ge \abs L$ is small,  there is an isomorphism of posets $\invtildestarof\kappa\monster\cong\invtildestarof\kappa{\RVsesof_*(\monster)}$. If $\otimes$ respects $\doms$ in $\RVsesof_*(\monster)$, then the same holds in $\monster$, and the above is an isomorphism of monoids.
\end{pr}
\begin{proof}
Adapt the proofs of Lemma~\ref{lemma:sepbas}, Proposition~\ref{pr:rvot} and Corollary~\ref{co:rvredinf}, replacing Fact~\ref{fact:rvqe} and Fact~\ref{fact:seprv} by~\cite[Theorem~B]{Basarab_1991} and Fact~\ref{fact:seprvmix} respectively.
\end{proof}
  The assumptions of Proposition~\ref{pr:rvotmix} are satisfied in a number of cases of interest. Besides the algebraically closed case, we note the following.
  \begin{rem}\label{rem:finram}
    Every $\RVses_*$-expansion of a finitely ramified henselian valued field has enough saturated models.
  \end{rem}
  \begin{proof}
Finite ramification ensures  immediate extensions are precisely those where $\RVses_*$ does not change. 
By this and~\cite[Proposition~4.3]{flenner},  maximal immediate extensions are elementary, and by~\cite[Corollary~4.29]{vddvf} they are also unique. We may therefore  adapt the proof of Proposition~\ref{pr:maxsat}, replacing $\RVses$ with $\RVses_*$.
\end{proof}

\begin{rem}
 $\RVses_*$ may be viewed as a short exact sequence of abelian structures, each consisting of an inverse system of abelian groups. Since $\Gamma$ is torsion-free, this sequence is pure.\footnote{
Another way of seeing this is that, in a saturated enough model of $T$, the valuation map has a section, inducing a compatible system of angular components, i.e.~a splitting of  $\RVses_*$.} Hence, the results from Section~\ref{sec:ses} apply to this setting, e.g.\ by taking as $\mathcal F$ the family of all pp formulas.
\end{rem}
If $\key$ eliminates imaginaries,  we can get rid of those arising from $\mathcal F$ and obtain a product decomposition. We state a special case as an example application of the results above. We thank the referee for pointing out the ``moreover'' part.
\begin{co}\label{co:wittv}
  In the theory of the Witt vectors over $\mathbb F_\primo^\mathrm{alg}$, the domination monoid is well defined. If $\kappa$ is a small infinite cardinal, then $
  \invtildestarof\kappa{\monster}\cong \invtildestarof\kappa{\keyof(\monster)}\times\invtildestarof\kappa{\Gamof(\monster)}\cong \hat \kappa\times \mathscr P_{\le \kappa}(\icsg(\Gamof(\monster)))$. Moreover, $\invtilde\cong \hat \omega\times \mathscr P_{<\omega}(\icsg(\Gamof(\monster)))$.
\end{co}

\begin{proof}The residue field $\key$ is fully embedded. Moreover, $\key_n=W_n(\key)^\times$ for each $n$, where $W_n(\key)$ is the truncated ring of Witt vectors over $\key$, and $\key_n$ is in definable bijection with $\key^{n-1}\times \key^\times$ (cf.\ \cite[Corollary~1.62 and Proposition 1.67]{touchardburden}). The computation of   $\invtildestarof\kappa{\monster}$ follows. As for $\invtilde$, using discreteness of the value group it is possible to build a pro-definable surjection $\K\to\key^\omega$~\cite[proof of Remark~3.23]{touchardburden}; together with the argument above, this gives the ``moreover'' part.
\end{proof}

\begin{rem}
The product decomposition fails for finitary types: the surjection $\K\to\key^\omega$  yields a $1$-type in $\K$ dominating the type of an infinite independent $\key$-tuple.
\end{rem}
However, finitisation is possible in the case of the $\primo$-adics.
\begin{co}[{Theorem~\ref{athm:qp}}]\label{co:qp}
  Let $T$ be a complete $\set\Gam$-expansion of 
  $\Th(\mathbb Q_\primo)$.
  There is an isomorphism of posets $\invtilde\cong\invtildeof{\Gamof(\monster)}$. If $\otimes$ respects $\doms$ in $\Gamof(\monster)$, then the same holds in $\monster$, and the above is also an isomorphism of monoids. In particular, in $\Th(\mathbb Q_\primo)$,  $\otimes$ respects $\doms$, and $
    (\invtilde, \otimes, \doms)\cong(\pfin(\icsg(\Gamof(\monster))), \cup, \supseteq)$.
\end{co}
\begin{proof}
  By Remark~\ref{rem:finram} we may apply Proposition~\ref{pr:rvotmix}. Since each $\key_n$ is finite, each $\RV_n$ is a finite cover of $\Gam$, so each element of $\RV_n$ is interalgebraic with an element of $\Gam$. Thus if $p(x,z)\in\invtypes(\monster, M_0)$, where $x$ is a tuple of $\K$-variables and $z$ a tuple of $\RVses_*$-variables, and if $ac\models p$, then $\dim_{\mathbb Q}(\mathbb Q\Gamof(\dcl(\monster(ac)))/\mathbb Q\Gamof(\monster))\leq|xz|$ by the Abhyankar inequality, so there is a finitary invariant type in $\Gamma$ which is interalgebraic with the type $q(y,t)\domeq p$ found in Proposition~\ref{pr:rvotmix}.
We conclude by Proposition~\ref{pr:comptransfer}. The ``in particular'' part then follows from Corollary~\ref{co:itro}.
\end{proof}
 The infinite ramification case remains open.
\begin{prob}
Compute  $\invtilde$ in an infinitely ramified mixed characteristic hen\-se\-lian valued field.  
\end{prob}

\section{D-henselian valued fields with many constants}\label{sec:diff}
Here we deal with certain differential valued fields.  As the proofs are adaptations of those in Section~\ref{sec:bvf},  we give sketches and leave it to the reader to fill in the details.

We let $T$ be a complete theory with sorts $\K, \key, \Gam, \RV$, as in Section~\ref{sec:bvf}, naturally interpreted, and use the notation $\RVses$.
The fields $\key$ and $\K$ have characteristic $0$ and  both carry a derivation $\partial$ (denoted by the same symbol), commuting with the residue map.
The valued differential field $\K$ is \emph{monotone}, i.e.~$v(\partial x)\ge v(x)$,  has \emph{many constants}\footnote{Here we follow the terminology of~\cite{adh}. In~\cite{scarefrob}, this condition is called having \emph{enough constants}.}, i.e.~for every $\gamma\in \Gam$ there is $x\in \K$ with $\partial x=0$ and $v(x)=\gamma$, and is \emph{D-henselian}, i.e.~the following holds. If $P(X)\in \mathcal O\set{X}=\mathcal O[\partial^i X]_{i\in \omega}$ is a differential polynomial over the valuation ring $\mathcal O$,  and $a\in \mathcal O$ is such that $v(P(a))>0$ and for some $i$ we have $v(\operatorname{d}P/\operatorname{d}(\partial^i X))(a)=0$, then there is $b\in \mathcal O$ such that $P(b)=0$ and $v(a-b)>0$.
The family of sorts  $\RVses$ may carry additional structure.

  The derivation  $\partial$ on $\K$ induces a map $\partial_\RV$ on $\RV$ which, for all $\gamma\in \Gam$, fixes $v\inverse(\gamma)\cup \set 0$ setwise, defined by $\partial_\RV(\rv(x))=\rv(\partial(x))$ iff $v(\partial(x))=v(x)$, and $\partial_\RV(\rv(x))=0$ otherwise, which extends the derivation $\partial$ on $\key$.

By~\cite[Theorem~6.4 and Corollary~5.8]{scarefrob} (see also~\cite[Corollary~8.3.3]{adh}) the theory $T$ given by the list of properties above (in a fixed language) eliminates $\K$-quantifiers.

\begin{pr}\label{pr:diffesmm}
The theory  $T$ has enough saturated maximal models.
\end{pr}
\begin{proof}[Proof sketch]
  By~\cite[Remark~6.2]{scarefrob}, $k$ is linearly surjective  in the terminology of~\cite{adh}, so by~\cite[Theorem~7.4.3]{adh} $T$ has uniqueness of maximal immediate extensions. The maximal immediate extension $N$ of $M$ is monotone and D-henselian by~\cite[Lemma~6.3.5 \& Theorem~7.4.3]{adh} with many constants.  As $T$ eliminates $\K$-quantifiers, $M\prec N$, so the proofs of Proposition~\ref{pr:maxsat} and Corollary~\ref{co:maxsat} may be adapted.
\end{proof}

\begin{thm}\label{thm:diff}
  Let $\kappa$ be a small infinite cardinal. 
There is an isomorphism of posets $\invtildestarof{\kappa}\monster\cong\invtildestarof{\kappa}{\RVsesof(\monster)}$. If $\otimes$ respects $\doms$ in $\RVsesof(\monster)$, then the same holds in $\monster$, and the above is also an isomorphism of monoids.
\end{thm}
\begin{proof}[Proof sketch]
By elimination of $\K$-quantifiers, $\RVsesof(M)$ is fully embedded in $M$. If we replace ``polynomial'' by ``differential polynomial'', $\Kof(M)[a]$ by $\Kof(M)\set{a}$, and so on, in the statements of Lemma~\ref{lemma:sepbas} and Proposition~\ref{pr:rvot}, essentially the same proofs go through. We can then conclude as in the proof of Corollary~\ref{co:rvredinf}.
\end{proof}

\begin{lemma}\label{lemma:constants}
$\partial_\RV$ is definable from the the short exact sequence structure, the differential field structure on $\key$, and a predicate for $\constants\coloneqq\set{c\in \RV\mid \partial_\RV(c)=0}$.
\end{lemma}
\begin{proof}
 Suppose  $a\in \RV$ and $v(a)\notin\set{0, \infty}$. Since $\K$ has many constants, there is $c\in \RV(M)$ with $\partial_\RV(c)=0$ and $v(c)=v(a)$. Then $a/c\in \keyof(\monster)$, and $\partial_\RV(a)=c\partial(a/c)$. Because this does not depend on the choice of $c$, the function $y=\partial_\RV(x)$ is $\emptyset$-definable by the formula $\phi(x,y) \coloneqq \exists z\in \constants\;((v(z)=v(x))\land  ( y=z\partial(x/z)))$.
\end{proof}
If $L$  had a section of the valuation, or an angular component compatible with $\partial$, we could recover $\constants$ from the constant field of $\key$, and conclude by (the \textasteriskcentered-type version of) Remark~\ref{rem:split}. Yet, the absence of definable splitting is not a serious obstacle. For simplicity, we only give a result in the model companion $\mathsf{VDF}_\mathcal{EC}$.

\begin{thm}[{Theorem~\ref{athm:vdf}}]\label{thm:vdfec}
  In  $\mathsf{VDF}_\mathcal{EC}$, for every small infinite cardinal $\kappa$, the monoid $\invtildestarof\kappa\monster$ is well-defined, and we have isomorphisms
  \[
    \invtildestarof\kappa\monster\cong \invtildestarof\kappa{\keyof(\monster)}\times \invtildestarof\kappa{\Gamof(\monster)}\cong \prod_{\delta(\monster)}^{\le\kappa} \hat \kappa\times \mathscr P_{\le\kappa}(\icsg(\Gamof(\monster)))
  \]
   where  $\delta(\monster)$ is a cardinal, and $\prod_{\delta(\monster)}^{\le\kappa} \hat \kappa$ denotes the submonoid of $\prod_{\delta(\monster)} \hat \kappa$ consisting of $\delta(\monster)$-sequences with support of size at most $\kappa$.
\end{thm}
\begin{proof}
  By Theorem~\ref{thm:diff} we reduce to $\RVses$.  Let $L_\constants\coloneqq L_\mathrm{ab}\cup \set{\constants}$, with $\constants$ a unary predicate.
Expand the language of $\RVses$ by a predicate $\constants$ on each sort, interpreted as the constants in both $\key$ and $\RV$ and as the full $\Gamma$ in $\Gamma$, obtaining a short exact sequence of $L_\constants$-abelian structures\footnote{To be precise, of abelian structures augmented by an absorbing element. See Remark~\ref{rem:variants}.}, expanded by the differential field structure on $\key$ and the order on $\Gamma$.  
 By  Lemma~\ref{lemma:constants}, we may apply the material from Section~\ref{sec:ses}, say by taking as a fundamental family that of all pp $L_\constants$-formulas, provided we show that $\RVses$ is pure. If $M\models \mathsf{VDF}_\mathcal{EC}$ is $\aleph_1$-saturated then, since $M$ has many constants, we may find a section $s\from \Gamma(M)\to \RVof(M)$ of the valuation with image included in $\constants(\RVof(M))$. Hence the short exact sequence $\RVsesof(M)$ of $L_\constants$-abelian structures splits, so is pure by Remark~\ref{rem:split}.  Since $\key$ is a model of $\mathsf{DCF}_0$, which eliminates imaginaries, we may get rid of the auxiliary sorts $\mathrm{A}_\phi$. We conclude by Corollary~\ref{co:isro} and the fact that $\mathsf{DCF}_0$ is $\omega$-stable multidimensional (see~\cite[Section~5]{invbartheory} for the relation between our setting and that of domination via forking in stable theories). 
\end{proof}
\begin{rem}\label{rem:vdfnofin}
  In $\mathsf{VDF}_\mathcal{EC}$, finitisation is not to be expected (e.g.~by~\cite[Proposition~4.2]{eivdf}), and in fact not possible: one may construct a $1$-type $p\in \invtypes_{\K}(\monster)$ with  $((v\circ\partial^n)_* p)_{n\in \omega}$  non-algebraic and pairwise $\wort$, hence not domination-equivalent. 
\end{rem}
Computing the image of the home sort in finitely many variables seems difficult.
\begin{rem}\label{rem:vdifff}
   Most arguments in this section may be adapted to $\sigma$-henselian valued difference fields of residue characteristic $0$. An analogue of Theorem~\ref{thm:diff} goes through, using quantifier reduction to $\RVses$ and a $\sigma$-Kaplansky theory yielding uniqueness and elementarity of maximal immediate extensions~\cite[Theorem~5.8  and Theorem~7.3]{durhanonay}.
 In every completion of the model companion of the \emph{isometric} case (see~\cite{bmsfrob}), in sufficiently saturated models there is a section of the valuation with values in the fixed field,  hence
 one may obtain the decomposition $\invtildestarof\kappa\monster\cong \invtildestarof\kappa{\keyof(\monster)}\times \invtildestarof\kappa{\Gamof(\monster)}$, by regarding $\RVses$ as a pure short exact sequence of $\mathbb Z[\sigma]$-modules, and using elimination of imaginaries in $\mathsf{ACFA}_0$. The same goes through in the \emph{multiplicative} setting, provided that,  in the notation of~\cite{pal},  $\rho$ is transcendental. This applies, e.g., to the model companion of the \emph{contractive} case (see~\cite{azgin}).
\end{rem}

\bibliographystyle{abbrv}

\end{document}